\documentclass[11pt]{amsart}

\usepackage[utf8]{inputenc}
\usepackage[english]{babel}
\usepackage[T1]{fontenc}

\usepackage{amsthm,amsmath,amssymb}
\usepackage{stmaryrd}
\usepackage{multirow}

\usepackage{graphicx}
\usepackage{tikz}
\usetikzlibrary{matrix,arrows}
\usetikzlibrary{positioning}
\usetikzlibrary{fit}
\usetikzlibrary{patterns}

\usepackage[colorlinks=true,citecolor=black,linkcolor=black,urlcolor=blue]{hyperref}

\newcommand{\dbrac}[1]{{\llbracket #1 \rrbracket}} 	% double brackets
\newcommand{\sep}{{\colon}}                         % Separator
\newcommand{\boks}[2]{({#1, #2})}   % A box with the old
\newcommand{\id}{\mathrm{id}}

\DeclareMathOperator{\des}{des}

\newcommand{\pattern}[4]{										% mesh pattern
  \raisebox{0.6ex}{
  \begin{tikzpicture}[scale=0.35, baseline=(current bounding box.center), #1]
    \foreach \x/\y in {#4}
      \fill[pattern=north east lines] (\x,\y) rectangle +(1,1);
    \draw (0.01,0.01) grid (#2+0.99,#2+0.99);
    \foreach \x/\y in {#3}
      \filldraw (\x,\y) circle (6pt);
  \end{tikzpicture}}
}

\theoremstyle{plain}
\newtheorem{theorem}{Theorem}
\newtheorem{lemma}[theorem]{Lemma}

\newtheorem{proposition}[theorem]{Proposition}

\newtheorem{observation}[theorem]{Observation}

\theoremstyle{definition}
\newtheorem{definition}[theorem]{Definition}
\newtheorem{example}[theorem]{Example}
\newtheorem{conjecture}[theorem]{Conjecture}

\theoremstyle{remark}

\title[Wilf-classification of mesh patterns of short length]{\bf Wilf-classification\\ of mesh patterns of short length}

\author[Hilmarsson,
        Jónsdóttir,
        Sigurðardóttir,
        Viðarsdóttir,
        Ulfarsson]{
        Ísak Hilmarsson,
        Ingibjörg Jónsdóttir,
        Steinunn Sigurðardóttir,
        Lína Viðarsdóttir and
        Henning Ulfarsson}
\address{School of Computer Science, Reykjavik University,
Menntavegi 1, \newline 101 \mbox{Reykjavík}, Iceland}
\email{henningu@ru.is \and
       ingibjorgj08@ru.is \and
       isakh08@ru.is \and
       sigridurlv08@ru.is \and
       steinunns08@ru.is}

\thanks{Research of the first four authors was partially supported by grant 090038013 from the Icelandic Research Fund. Research of the fifth author was
supported by the same grant as well as grant 141761051.}

\begin{document}

\begin{abstract} This paper starts the Wilf-classification of mesh patterns of
length $2$. Although there are initially $1024$ patterns to consider we
introduce automatic methods to reduce the number of potentially different
Wilf-classes to at most $65$. By enumerating some of the remaining classes we
bring that upper-bound further down to $56$. Finally, we conjecture that the
actual number of Wilf-classes of mesh patterns of length $2$ is $46$.

\end{abstract}

\maketitle

\thispagestyle{empty}

\section{Introduction}
Let $n$ be a non-negative integer. A \emph{permutation} is a bijection from the
set $\{1, 2, \dotsc, n \}$ to itself. The permutation that maps $i$ to $\pi_i$
will be written as the word $\pi=\pi_1\pi_2 \dotsm \pi_n$. Let $S_n$ be the set
of all permutations of length $n$.

A (\emph{classical permutation}) \emph{pattern} is a permutation $p \in S_k$.
The pattern $312 \in S_3$ can be drawn as follows, where the horizontal lines
represent the values and the vertical lines denote the positions in the pattern.
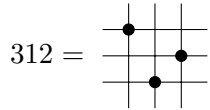
\[
312 = \pattern{scale=1}{3}{1/3,2/1,3/2}{}
\]

We say that a pattern $p$ \emph{occurs} in a permutation $\pi \in S_n$ if there
is a subsequence of $\pi$ whose letters are in the same relative order of size
as the letters of $p$. This sequence is called an \emph{occurrence} of the
pattern $p$ in the permutation $\pi$. If a pattern occurs in a permutation we
say that the permutation \emph{contains} the pattern. For example, the
permutation $25134$ contains the pattern $312$ as the subsequence $534$. The
diagram below shows the permutation where points corresponding to the occurrence
of the pattern have been circled.
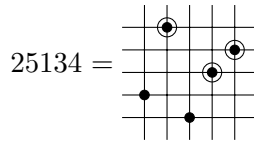
\[
25134 =
 \begin{tikzpicture}[scale=0.3, baseline=(current bounding box.center)]
    \draw (0.01,0.01) grid (5+0.99,5+0.99);
  \filldraw (1,2) circle (6pt);
  \filldraw (2,5) circle (6pt);
  \filldraw (3,1) circle (6pt);
  \filldraw (4,3) circle (6pt);
  \filldraw (5,4) circle (6pt);
  % draw three circles around the points
  \draw (2,5) circle (12pt);
  \draw (4,3) circle (12pt);
  \draw (5,4) circle (12pt);
 \end{tikzpicture}
\]
A permutation that does not contain a pattern is said to \emph{avoid} the
pattern. For example, a permutation $\pi \in S_n$ avoids the pattern $231$ if
there do not exist $1 \leq i < j < k \leq n$ with $\pi(k) < \pi(i) < \pi(j)$. An
example of a permutation that avoids the pattern $231$ is the permutation
$51423$.
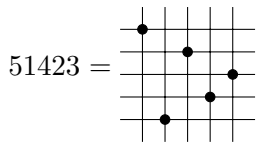
\[
51423=
 \begin{tikzpicture}[scale=0.3, baseline=(current bounding box.center)]
    \draw (0.01,0.01) grid (5+0.99,5+0.99);
  \filldraw (1,5) circle (6pt);
  \filldraw (2,1) circle (6pt);
  \filldraw (3,4) circle (6pt);
  \filldraw (4,2) circle (6pt);
  \filldraw (5,3) circle (6pt);
 \end{tikzpicture}
\]
Given a pattern $p$ we let $S_n(p)$ be the set of permutations of length $n$
that avoid $p$.

One of the primary questions in the theory of permutation patterns is that of
Wilf-equivalence: Given two patterns $p$ and $q$, are the sizes of the sets
$S_n(p)$ and $S_n(q)$ equal for all $n$? Patterns for which the answer
is ``yes'' are called \emph{Wilf-equivalent}. A \emph{Wilf-class} is a maximal
set of patterns (necessarily of the same length) that are all Wilf-equivalent.
The process of sorting patterns into classes by Wilf-equivalence is called
\emph {Wilf-classification}. Classical patterns of length $3$ were
Wilf-classified by Knuth~\cite{MR0378456}, who showed that the number of
permutations avoiding each classical pattern of length $3$ is given by the
Catalan numbers. Permutations avoiding more than one pattern have also been
studied. Simion and Schmidt~\cite{MR829358} Wilf-classified all sets of
classical patterns of length $3$.

The purpose of this paper is to start the Wilf-classification of mesh patterns.
Mesh patterns, whose definition we review below, provide a common extension of
several previous generalizations of classical patterns. We show that the $16$
mesh patterns of length $1$ belong to $4$ different Wilf-classes. The
classification of the $1024$ mesh patterns of length $2$ would be very tedious
to to by hand without resorting to the usual $D_8$ symmetries of patterns. This
however only brings the number of patterns down to $186$. In
Lemma~\ref{lem:shading} we give a sufficient condition for when a shading can be
added to a mesh pattern. This cuts the number of patterns in half, down to $87$.
Two other operations allow us to bring that number down to $65$. We then use
conventional tools of combinatorics, like generating functions, bijective maps,
etc., to achieve the upper bound of $56$ on the number of Wilf-classes. We
conjecture that the actual number of classes is $46$.

\section{An overview of generalized patterns}
Several generalizations of classical patterns have been introduced. The first
extension relevant to us are \emph{vincular patterns}, defined by Babson and
Steingr\'imsson~\cite{MR1758852}. These patterns can require letters in a
permutation to be adjacent. For example the pattern $\pattern{scale=0.45}{ 3 }{
1/2, 2/1, 3/3 }{2/0, 2/1, 2/2, 2/3}$ requires the letters corresponding to $1$
and $3$ in a permutation to be adjacent. Graphically this means that no points
can be in the shaded area. This pattern occurs in the permutation
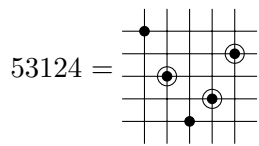
\[
53124 =
 \begin{tikzpicture}[scale=0.3, baseline=(current bounding box.center)]
    \draw (0.01,0.01) grid (5+0.99,5+0.99);
  \filldraw (1,5) circle (6pt);
  \filldraw (2,3) circle (6pt);
  \filldraw (3,1) circle (6pt);
  \filldraw (4,2) circle (6pt);
  \filldraw (5,4) circle (6pt);
  % draw three circles around the points
  \draw (2,3) circle (12pt);
  \draw (4,2) circle (12pt);
  \draw (5,4) circle (12pt);
 \end{tikzpicture}
\]
because $53124$ contains the classical pattern $213$ as the subsequence $324$
and because the letters in the permutation that correspond to $1$ and $3$ in
the pattern, i.e., $2$ and $4$, are adjacent in the permutation.

The permutation $52314$ avoids the pattern $\pattern{scale=0.45}{3}{1/1, 2/2,
3/3}{2/0, 2/1, 2/2, 2/3}$, since the only occurrence of the classical pattern
$123$ is the subsequence $234$ and the letters $3$ and $4$ are not adjacent in
the permutation.
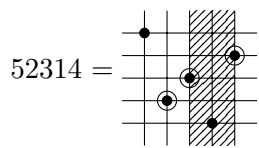
\[
52314 =
 \begin{tikzpicture}[scale=0.3, baseline=(current bounding box.center)]
  \draw (0.01,0.01) grid (5+0.99,5+0.99);
  \fill[pattern=north east lines] (3,0) rectangle +(2,6);
  \filldraw (1,5) circle (6pt);
  \filldraw (2,2) circle (6pt);
  \filldraw (3,3) circle (6pt);
  \filldraw (4,1) circle (6pt);
  \filldraw (5,4) circle (6pt);
  % draw three circles around the points
  \draw (2,2) circle (12pt);
  \draw (3,3) circle (12pt);
  \draw (5,4) circle (12pt);
 \end{tikzpicture}
\]
Vincular patterns of length $3$ with one shaded column were Wilf-classified by
Claesson~\cite{MR1857258}. He showed that the number of permutations avoiding
eight of the vincular patterns is given by the Bell numbers and the remaining
four give the Catalan numbers. Several of his results are a consequence of
Lemma~\ref{lem:shading} below.
\newline

\emph{Bivincular patterns} are a natural extension of vincular patterns where we
may also put constraints on the values in a permutation. Bivincular patterns
where first introduced by Bousquet-M\'elou et al.~\cite{MR2652101}. The pattern
$p = \pattern{scale=0.5}{ 3 }{ 1/1, 2/3, 3/2 }{0/2, 1/0, 1/1, 1/2, 1/3, 2/2,
3/2}$ requires the letters corresponding to $1$ and $3$ in a permutation to be
adjacent in position and the letters corresponding to $2$ and $3$ in a
permutation to be adjacent in size. The permutation
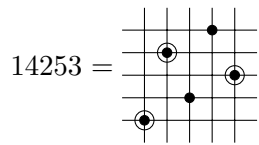
\[
14253 =
 \begin{tikzpicture}[scale=0.3, baseline=(current bounding box.center)]
    \draw (0.01,0.01) grid (5+0.99,5+0.99);
  \filldraw (1,1) circle (6pt);
  \filldraw (2,4) circle (6pt);
  \filldraw (3,2) circle (6pt);
  \filldraw (4,5) circle (6pt);
  \filldraw (5,3) circle (6pt);
  % draw three circles around the points
  \draw (1,1) circle (12pt);
  \draw (2,4) circle (12pt);
  \draw (5,3) circle (12pt);
 \end{tikzpicture}
\]
contains the pattern $p$ because it contains the classical pattern $132$ as the
subsequence $143$ and letters $1$ and $4$ are adjacent in position and letters
$3$ and $4$ are adjacent in size. On a diagram, we observe that there are no
points in the shaded areas.
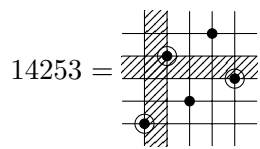
\[
14253 =
 \begin{tikzpicture}[scale=0.3, baseline=(current bounding box.center)]
    \draw (0.01,0.01) grid (5+0.99,5+0.99);
    \fill[pattern=north east lines] (0,3) rectangle +(6,1);
    \fill[pattern=north east lines] (1,0) rectangle +(1,6);
  \filldraw (1,1) circle (6pt);
  \filldraw (2,4) circle (6pt);
  \filldraw (3,2) circle (6pt);
  \filldraw (4,5) circle (6pt);
  \filldraw (5,3) circle (6pt);
  % draw three circles around the points
  \draw (1,1) circle (12pt);
  \draw (2,4) circle (12pt);
  \draw (5,3) circle (12pt);
 \end{tikzpicture}
\]

However the permutation $12543$ contains one occurrence of the classical pattern
$132$ as the subsequence $253$ but $5$ and $3$ are not adjacent in size, so that
is not an occurrence of the bivincular pattern.
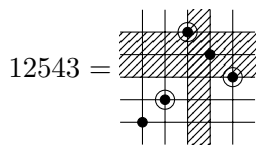
\[
12543 =
 \begin{tikzpicture}[scale=0.3, baseline=(current bounding box.center)]
    \draw (0.01,0.01) grid (5+0.99,5+0.99);
    \fill[pattern=north east lines] (3,0) rectangle +(1,6);
  \fill[pattern=north east lines] (0,3) rectangle +(6,2);
  \filldraw (1,1) circle (6pt);
  \filldraw (2,2) circle (6pt);
  \filldraw (3,5) circle (6pt);
  \filldraw (4,4) circle (6pt);
  \filldraw (5,3) circle (6pt);
  % draw three circles around the points
  \draw (2,2) circle (12pt);
  \draw (3,5) circle (12pt);
  \draw (5,3) circle (12pt);
 \end{tikzpicture}
\]
Bivincular patterns of length $2$ and $3$ were Wilf-classified by Parviainen
in~\cite{RB}. Some of his results are a consequence of Lemma~\ref{lem:shading}
below.
\newline

Mesh patterns where first introduced by Br\"and\'en and
Claesson~\cite{MR2795782}, as a further extension of bivincular patterns, that
also subsumes barred patterns with a single bar~\cite{W90}, and interval
patterns~\cite{MR2422304}. A pair $(\tau,R)$, where $\tau$ is a permutation in
$S_k$ and $R$ is a subset of $\dbrac{0,k} \times \dbrac{0,k}$, where
$\dbrac{0,k}$ denotes the interval of the integers from $0$ to $k$, is a
\emph{mesh pattern} of length $k$.
\newline

Let $\boks{i}{j}$ denote the box whose corners have coordinates $(i,j), (i,j+1),
(i+1,j+1)$ and $(i+1,j)$. An example of a mesh pattern is the classical pattern
$312$ along with $R = \{\boks{1}{2},\boks{2}{1}\}$. We draw this by shading the
boxes in $R$
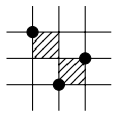
\[
\pattern{scale=1}{ 3 }{ 1/3, 2/1, 3/2 }{1/2, 2/1}
\]
The permutation $521643$ contains this pattern, see below
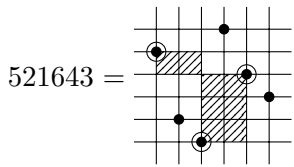
\[
521643 =
 \begin{tikzpicture}[scale=0.3, baseline=(current bounding box.center)]
    \draw (0.01,0.01) grid (6+0.99,6+0.99);
    \fill[pattern=north east lines] (1,4) rectangle +(2,1);
  \fill[pattern=north east lines] (3,1) rectangle +(2,3);
  \filldraw (1,5) circle (6pt);
  \filldraw (2,2) circle (6pt);
  \filldraw (3,1) circle (6pt);
  \filldraw (4,6) circle (6pt);
  \filldraw (5,4) circle (6pt);
  \filldraw (6,3) circle (6pt);
  % draw three circles around the points
  \draw (1,5) circle (12pt);
  \draw (3,1) circle (12pt);
  \draw (5,4) circle (12pt);
 \end{tikzpicture}
\]
The permutation has an occurrence of the mesh pattern as the subsequence $514$,
since it forms the classical pattern $312$ and there are no points in the shaded
areas.

Let's now look at the permutation $\pi = 32145$. This permutation avoids the
pattern
$(123,\{\boks{0}{1},\boks{1}{0},\boks{2}{2}\})=\pattern{scale=0.5}{3}{1/1, 2/2,
3/3}{0/1, 1/0, 2/2}$, because for all occurrences of the classical pattern $123$
there is at least one point in at least one of the shaded boxes. For example,
the subsequence $245$ in $\pi$ is an occurrence of the classical pattern $123$
but not of the mesh pattern since the point representing $1$ is in one of the
shaded areas. This can be seen on the following diagram.
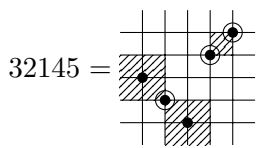
\[
32145 =
 \begin{tikzpicture}[scale=0.3, baseline=(current bounding box.center)]
    \draw (0.01,0.01) grid (5+0.99,5+0.99);
    \fill[pattern=north east lines] (4,4) rectangle +(1,1);
  \fill[pattern=north east lines] (0,2) rectangle +(2,2);
  \fill[pattern=north east lines] (2,0) rectangle +(2,2);
  \filldraw (1,3) circle (6pt);
  \filldraw (2,2) circle (6pt);
  \filldraw (3,1) circle (6pt);
  \filldraw (4,4) circle (6pt);
  \filldraw (5,5) circle (6pt);
  % draw three circles around the points
  \draw (2,2) circle (12pt);
  \draw (4,4) circle (12pt);
  \draw (5,5) circle (12pt);
 \end{tikzpicture}
\]

\section{Operations preserving Wilf-equivalence}
The number of mesh patterns of length $n$ is $n! \cdot 2^{(n+1)^2}$ so already
for $n = 2$ we have $1024$ patterns. This makes the Wilf-classification
very tedious to do by hand. We therefore review known operations (or symmetries)
that preserve Wilf-equivalence as well as introducing new ones. This will allow
us to automatically bring the number of patterns to consider down to $65$.

The first known operations are the symmetries reverse, complement and inverse.
For a given mesh pattern $(\tau,R)$ of length $n$, we define
\begin{equation*}
    (\tau,R)^r = (\tau^r, R^r), \quad
    (\tau,R)^c = (\tau^c,R^c), \quad
    (\tau,R)^i = (\tau^i,R^i),
\end{equation*}
where $\tau^r$ is the usual reverse of the permutation $\tau$, $\tau^c$ the
usual complement, $\tau^i$ the usual inverse, and
\begin{align*}
&R^r = \{ \boks{n-x}{y} \sep \boks{x}{y} \in R \}, \\
&R^c = \{ \boks{x}{n-y} \sep \boks{x}{y} \in R \}, \\
&R^i = \{ \boks{y}{x}   \sep \boks{x}{y} \in R \}.
\end{align*}
Hence, reverse is a reflection around the vertical center line, complement is a
reflection around the horizontal center line and inverse is the reflection
around the southwest to northeast diagonal. Figure~\ref{fig:meshsyms} is an
example of the use of these symmetries on the pattern $p
=(312,\{\boks{0}{1},\boks{1}{3},\boks{2}{2}\})$.
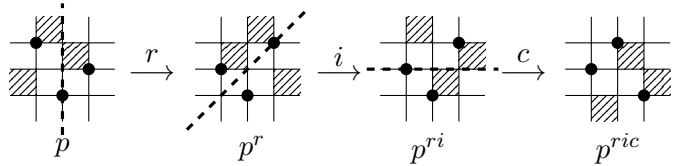
\begin{figure}[hb]
 \begin{tikzpicture}[scale=0.35]
  \begin{scope}
    % Make the grid
    \draw (0.01,0.01) grid (3+0.99,3+0.99);

    % Make the shaded boxes
    \fill[pattern=north east lines] (0,1) rectangle +(1,1);
    \fill[pattern=north east lines] (1,3) rectangle +(1,1);
    \fill[pattern=north east lines] (2,2) rectangle +(1,1);

  % Make the pattern
  \filldraw (1,3) circle (6pt);
  \filldraw (2,1) circle (6pt);
  \filldraw (3,2) circle (6pt);

  % Make the dashed line to show reverse

  \draw [dashed,very thick,black] (2,-0.5) -- (2,4.5);

  \end{scope}

  \node at (2,-1)  {$p$};
  \node at (5.4,2.5) {$r$};
  \node at (5.4,1.8 ) {$\longrightarrow$};

  \begin{scope}[xshift = 200]

    % Make the grid
    \draw (0.01,0.01) grid (3+0.99,3+0.99);

    % Make the shaded boxes
    \fill[pattern=north east lines] (1,2) rectangle +(1,1);
    \fill[pattern=north east lines] (2,3) rectangle +(1,1);
    \fill[pattern=north east lines] (3,1) rectangle +(1,1);

  % Make the pattern
  \filldraw (1,2) circle (6pt);
  \filldraw (2,1) circle (6pt);
  \filldraw (3,3) circle (6pt);

  % Make the dashed line to show reverse

    \draw [dashed,very thick,black] (-0.3,-0.3) -- (4.3,4.3);

  \end{scope}

    \node at (12.5,2.5) {$i$};
  \node at (12.5,1.8 ) {$\longrightarrow$};
  \node at (9.2,-1) {$p^r$};

  \begin{scope}[xshift = 400]

    % Make the grid
    \draw (0.01,0.01) grid (3+0.99,3+0.99);

    % Make the shaded boxes
    \fill[pattern=north east lines] (1,3) rectangle +(1,1);
    \fill[pattern=north east lines] (2,1) rectangle +(1,1);
    \fill[pattern=north east lines] (3,2) rectangle +(1,1);

  % Make the pattern
  \filldraw (1,2) circle (6pt);
  \filldraw (2,1) circle (6pt);
  \filldraw (3,3) circle (6pt);

  % Make the dashed line to show reverse

    \draw [dashed,very thick,black] (-0.5,2) -- (4.5,2);

  \end{scope}

      \node at (19.5,2.5) {$c$};
  \node at (19.5,1.8 ) {$\longrightarrow$};
  \node at (15.9,-1) {$p^{ri}$};

  \begin{scope}[xshift = 600]

    % Make the grid
    \draw (0.01,0.01) grid (3+0.99,3+0.99);

    % Make the shaded boxes
    \fill[pattern=north east lines] (1,0) rectangle +(1,1);
    \fill[pattern=north east lines] (2,2) rectangle +(1,1);
    \fill[pattern=north east lines] (3,1) rectangle +(1,1);

  % Make the pattern
  \filldraw (1,2) circle (6pt);
  \filldraw (2,3) circle (6pt);
  \filldraw (3,1) circle (6pt);

  % Make the dashed line to show reverse

  \end{scope}

  \node at (23.1,-1) {$p^{ric}$};

 \end{tikzpicture}
\caption{Several symmetries of a mesh pattern} \label{fig:meshsyms}
\end{figure}
It is well-known that a permutation $\pi$ avoids a mesh pattern $p$ if and only
if the permutation $\pi^r$ avoids $p^r$. That is, the reverse operation
preserves Wilf-equivalence. The same applies for complement and inverse or any
composition of these three operations.

We now define the first of the new operations.

\begin{definition}
Let $\pi$ be a permutation (or a classical pattern) of length $n$. We define
the \emph{up-shift} of $\pi$, as $\pi^\uparrow$, where
\[
  \pi^\uparrow_i = (\pi_i \bmod n) + 1, \quad \text{for $1 \leq i \leq n$}.
\]
Let $p = (\tau, R)$ be a mesh pattern of length $n$. We define the
\emph{up-shift} of $p$ as the pattern $p^{\uparrow} = (\tau^\uparrow,
R^\uparrow)$, where
\[
  R^\uparrow = \{ \boks{a}{(b+1) \bmod (n+1)} \sep \boks{a}{b} \in R \}.
\]
\end{definition}

\begin{example}
This example shows the effect of the up-shift operation on a mesh pattern of
length $6$.
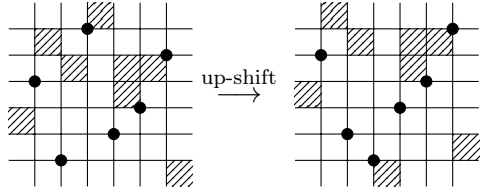
\[
\pattern{scale = 1}{6}{1/4,2/1,3/6,4/2,5/3,6/5}{0/2,2/4,1/5,3/6,4/3,4/4,5/4,6/0}
\overset{\text{up-shift}}{\longrightarrow}
\pattern{scale = 1}{6}{1/5,2/2,3/1,4/3,5/4,6/6}{0/3,2/5,1/6,3/0,4/4,4/5,5/5,6/1}
\]
\end{example}

\begin{proposition} \label{prop:ups}
Let $p=(\tau, R)$ be a pattern of length $n$ where the top line is shaded, that
is $\{(0,n), (1,n), \ldots, (n,n)\}\subseteq R$. Then a permutation $\pi$
avoids $p$ if and only if $\pi^{\uparrow}$ avoids $p^{\uparrow}$. That is,
up-shift preserves Wilf-equivalence for this kind of pattern.
\end{proposition}

\begin{proof}
Let $\pi$ be a permutation with an occurrence $\pi_{i_1}, \dots, \pi_{i_n}$ of a
pattern $p$ as described in the proposition. It is easy to see that
$\pi^{\uparrow}_{i_1}, \dots, \pi^{\uparrow}_{i_n}$ is an occurrence of
$p^{\uparrow}$ in $\pi^{\uparrow}$. Since the map $S_n \to S_n$, $\pi \mapsto
\pi^{\uparrow}$ is a bijection the claim follows.
\end{proof}

\begin{example}
Here is another example that shows the effect of up-shift on a mesh pattern of
length $6$.
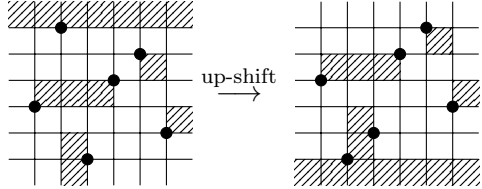
\[
\pattern{scale = 1}{6}{1/3,2/6,3/1,4/4,5/5,6/2}{0/6,1/3,1/6,2/0,2/1,2/3,2/6,3/3,3/6,4/6,5/4,5/6,6/2,6/6}
\overset{\text{up-shift}}{\longrightarrow}
\pattern{scale = 1}{6}{1/4,2/1,3/2,4/5,5/6,6/3}{0/0,1/0,1/4,2/0,2/1,2/2,2/4,3/0,3/4,4/0,5/0,5/5,6/0,6/3}
\]
According to Proposition~\ref{prop:ups} these two patterns are Wilf-equivalent.
\end{example}

Given a permutation $\pi \in S_n$, $0\pi$ is the permutation $\pi$ with $0$
prepended. Then define $\pi \boxplus 1 = (\pi_1 +1 \bmod(n+1))(\pi_2 +1
\bmod(n+1)) \cdots (\pi_n + 1 \bmod(n+1))$. For a given permutation $\lambda$ of
$\dbrac{0,n}$, $\lambda_0$ is the permutation in $S_n$ obtained by appending the
part of $\lambda$ to the left of $0$ to the other part, e.g., $14032_0 =3214$.

\begin{definition}[Ulfarsson~\cite{K}]
Let $\pi$ be a permutation (or a classical pattern) of length $n$. We define
the \emph{toric-shift} of $\pi$ as
\[
  \pi^t = (\pi^0 \boxplus 1)_0.
\]
Let $p = (\tau, R)$ be a mesh pattern of length $n$. We define the
\emph{toric-shift} of $p$ as the pattern $p^t = (\tau^t, R^t)$, where
\[
  R^t = \{ \boks{a + (n+1-\ell) \bmod (n+1)}{b + 1 \bmod (n+1)} \sep
    \boks{a}{b} \in R \}.
\]
Here $\ell$ is the position of the letter $n$ in the classical pattern $\tau$.
\end{definition}

\begin{example}
This example shows the effect of toric-shift on a mesh pattern of length $6$.
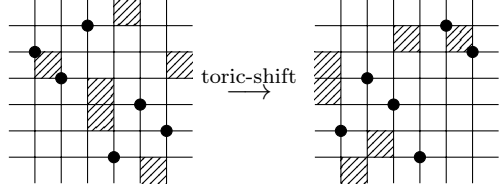
\[
\pattern{scale = 1}{6}{1/5,2/4,3/6,4/1,5/3,6/2}{1/4,3/2,3/3,4/6,5/0,6/4}
\overset{\text{toric-shift}}{\longrightarrow}
\pattern{scale = 1}{6}{1/2,2/4,3/3,4/1,5/6,6/5}{0/3,0/4,1/0,2/1,3/5,5/5}
\]
\end{example}

\begin{observation}
Let $p= (\tau, R)$ be a pattern of length $n$ where the top line is shaded,
that is $\{(0,n), (1,n), \ldots, (n,n)\}\subseteq R$. Recall from
Ulfarsson~\cite{K} that a permutation $\pi$ avoids $p$ if and only if $\pi^t$
avoids $p^t$. That is, toric-shift preserves Wilf-equivalence for this kind of
pattern.
\end{observation}

\begin{definition}
Two mesh patterns $p$ and $q$ are said to be \emph{coincident}, denoted by
$p \asymp q$, if for any permutation $\pi$, $\pi$ avoids $p$ if and only if
$\pi$ avoids $q$.
\end{definition}

Coincident patterns are obviously Wilf-equivalent.

\begin{observation}\label{obs:shading}
Let $R' \subseteq R$. Then any occurrence of $(\tau, R)$ in a permutation is an
occurrence of $(\tau, R')$.
\end{observation}

The next example will be generalized below into a powerful lemma that allows
adding more shaded boxes to a mesh pattern, while maintaining the coincidence
with the original pattern.

\begin{example}
Consider the mesh pattern $p=(12,\varnothing)=
\pattern{scale = 0.5}{2}{1/1,2/2}{}$. Let $u$ be the point $(1,1)$ and $v$ the
point $(2,2)$, in the pattern. We claim that the mesh pattern
$q=(12,\{\boks{0}{0}\})=\pattern{scale = 0.5}{2}{1/1,2/2}{0/0}$ is coincident to
$p$. Because of Observation~\ref{obs:shading} it suffices to show that if a
permutation contains $p$ it also contains $q$. Let $\pi$ be a permutation that
contains $p$ and consider a particular occurrence of it. Let $k$ be the number of
points in the box $\boks{0}{0}$ in $\pi$. If $k = 0$, it is clear that $\pi$
contains $q$ as well. If $k\geq 1$, then we can choose the leftmost (or the
lowest point), call it $d$, in the box $\boks{0}{0}$ and replace $u$ with $d$.
It is clear that the subsequence $ud$ satisfies the requirements of the mesh in
$q$. This can be interpreted as shading the box $\boks{0}{0}$.
Figure~\ref{fig:patteq} shows an example of this coincidence. In the left image
in the figure, the pattern $p$ can be found with the points $(7,5)$ as $u$ and
$(8,8)$ as $v$. There are points in the area at the lower left of $u$ and the
circled point, $(2,2)$, is the leftmost point in that area, and thus we denote
that point as $d$. In the right image, point $v$ is still $(8,8)$ but now $u$
has been replaced with $d$.
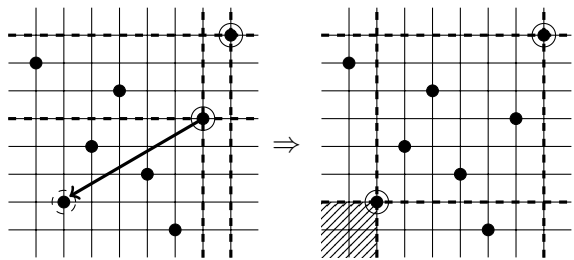
\begin{figure}[ht]
 \begin{tikzpicture}[scale=0.37, baseline=(current bounding box.center),auto]
 \begin{scope}
    \draw (0.01,0.01) grid (8+0.99,8+0.99);
  \filldraw (1,7) circle (6pt);
  \filldraw (2,2) circle (6pt);
  \filldraw (3,4) circle (6pt);
  \filldraw (4,6) circle (6pt);
  \filldraw (5,3) circle (6pt);
  \filldraw (6,1) circle (6pt);
  \filldraw (7,5) circle (6pt);
  \filldraw (8,8) circle (6pt);
  % draw the circles around the points
  \draw (7,5) circle (12pt);
  \draw (8,8) circle (12pt);
  \draw[dashed] (2,2) circle (12pt);

  %draw the arrow
  \draw[->,very thick] (7,5) -- (2.2,2.2);

  % The dashed lines

 \draw [dashed, very thick] (7,0) -- (7,9);
 \draw [dashed, very thick] (0,5) -- (9,5);

  \draw [dashed, very thick] (8,0) -- (8,9);
 \draw [dashed, very thick] (0,8) -- (9,8);

 \end{scope}

  \node at (10,4) {$\Rightarrow$};

  \begin{scope}[xshift=320]
    \draw (0.01,0.01) grid (8+0.99,8+0.99);
  \filldraw (1,7) circle (6pt);
  \filldraw (2,2) circle (6pt);
  \filldraw (3,4) circle (6pt);
  \filldraw (4,6) circle (6pt);
  \filldraw (5,3) circle (6pt);
  \filldraw (6,1) circle (6pt);
  \filldraw (7,5) circle (6pt);
  \filldraw (8,8) circle (6pt);

  % draw the circles around the points
  \draw (2,2) circle (12pt);
  \draw (8,8) circle (12pt);

  % The dashed lines

  \draw [dashed, very thick] (8,0) -- (8,9);
  \draw [dashed, very thick] (0,8) -- (9,8);

  \draw [dashed, very thick] (2,0) -- (2,9);
 \draw [dashed, very thick] (0,2) -- (9,2);

 % Shade 1 box
\fill[pattern=north east lines] (0,0) rectangle + (2,2);
\end{scope}
 \end{tikzpicture}
\caption{This picture shows two coincident patterns in the permutation
$72463158$} \label{fig:patteq}
 \end{figure}
\end{example}

This example generalizes to a new operation, introduced in
Lemma~\ref{lem:shading}, which preserves coincidence of mesh patterns.

\begin{lemma}[Shading Lemma] \label{lem:shading}
Let $(\tau,R)$ be a mesh pattern of length $n$ such that $\tau(i)=j$ and the
box $\boks{i}{j} \not\in R$. If all of the following conditions are satisfied:
\begin{enumerate}
  \item \label{lem:shading:req1} The box $\boks{i-1}{j-1}$ is not in $R$;
  \item \label{lem:shading:req2} At most one of the boxes $\boks{i}{j-1}$,
    $\boks{i-1}{j}$ is in $R$;
  \item \label{lem:shading:req3} If the box $\boks{\ell}{j-1}$ is in $R$
    (with $\ell \neq i-1,i$) then the box $\boks{\ell}{j}$ is also in $R$;
  \item \label{lem:shading:req4} If the box $\boks{i-1}{\ell}$ is in $R$
    (with $\ell \neq j-1, j$) then the box $\boks{i}{\ell}$ is also in $R$;
\end{enumerate}
then the patterns $(\tau,R)$ and $(\tau,R\cup\{\boks{i}{j} \})$ are coincident.
Analogous conditions determine if other boxes neighboring the point $(i,j)$ can
be added to $R$ while preserving the coincidence of the corresponding patterns.
\begin{center}
 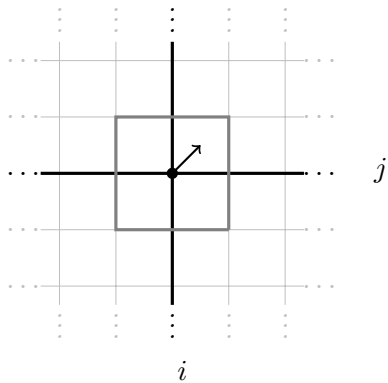
\begin{figure}[ht]
 \begin{tikzpicture}[scale = 0.75]

 % Make x-axis and y-axis
% \draw[->] (0,0) -- (0,6.5);
 %\draw[->] (0,0) -- (6.5,0);

 % Make the grid
 \draw[step=1cm,gray!50,very thin] (2/3,2/3) grid (5+1/3,5+1/3);

 % Make black cross in the middle
 \draw[-,very thick] (3,2/3) -- (3,5+1/3);
 \draw[-,very thick] (2/3,3) -- (5+1/3,3);
 \draw[-,very thick,gray] (4,2) -- (2,2) -- (2,4) -- (4,4) -- (4,2);

 \fill[black] (3,3) circle (0.1);

 % Draw the dots on the side
 \foreach \x in {1,2,4,5}{
   \node [gray!50,below] at (\x, 9.7/10) {$\vdots$};
   \node [gray!50,below] at (\x, 6+4/10) {$\vdots$};
}
 \foreach \y in {1,2,4,5}{
   \node [gray!50,below] at (4/10,\y + 0.21) {$\dots$};
   \node [gray!50,below] at (5+6.5/10,\y + 0.21) {$\dots$};
}
 % Draw black dots next to the cross
 \node [black,below] at (3, 9.7/10){$\vdots$};
 \node [black,below] at (3, 6+4/10){$\vdots$};

 % Draw black dots next to the cross
 \node [black,below] at (4/10, 3+0.21){$\dots$};
 \node [black,below] at (5+6.5/10, 3+0.21){$\dots$};

 % Draw the small lines on the axes
% \foreach \x in {1,2,3,4,5,6}
  % \draw (\x cm,1.5pt) -- (\x cm,-1.5pt) node[anchor = north]{};

% \foreach \y in {1,2,3,4,5,6}
 %  \draw (1.5pt,\y cm) -- (-1.5pt,\y cm) node[anchor = east]{};

 \draw[->,thick] (3,3) -- (3.5,3.5);

 \node [black,below] at (6.7,3.45) {$j$};
 \node [black,left]  at (3.45,-0.5) {$i$};

 \end{tikzpicture}
\caption{If the condition of the lemma are satisfied the box $\boks{i}{j}$ can
be shaded}
 \end{figure}
 \end{center}
\end{lemma}

\begin{proof}
In this proof we assume that the box $\boks{i-1}{j}$ is not in $R$.
According to Observation~\ref{obs:shading} we know that if a permutation $\pi$
contains the pattern $q = (\tau, R \cup \{ \boks{i}{j} \})$, it also contains
the pattern $p = (\tau,R)$.

Now we show that if a permutation $\pi$ contains the pattern $p$ it also
contains the pattern $q$. Assume we have a particular occurrence of $p$ in $\pi$
so the pattern point $(i,j)$ corresponds to a particular point $(i',j')$ in
$\pi$ and the box $\boks{i}{j}$ corresponds to a certain region $\mathcal{K}$ in
$\pi$ containing $k$ points. If $k = 0$ then we have an occurrence of $q$. If $k
\geq 1$ then let $(i'',j'')$ be the rightmost\footnote{If we had assumed that
the box $\boks{i}{j-1}$ was not in $R$ then we would have chosen the topmost
point in the region $\mathcal{K}$.} point in the region $\mathcal{K}$. This
amounts to shifting the line $x=i'$ to the line $x=i''$ and the line $y=j'$ to
the line $y=j''$. While we have now an empty region in the permutation
corresponding to the box $\boks{i}{j}$ in the pattern we need to make sure we
have not violated any of the requirements of the mesh $R$. The shifting of the
lines has the following effect on the other neighboring boxes of the point
$(i,j)$:
\begin{itemize}
  \item The box $\boks{i-1}{j}$ is shrunk from below, but extended to the
    right. It might contain extra points after the shifting, but that does not
    matter since it is not shaded.
  \item The box $\boks{i-1}{j-1}$ is extended up and to the right. It now
    contains the point $(i',j')$ (and possibly others) but this box is not a
    part of the shading, because of requirement~\eqref{lem:shading:req1}.
  \item The box $\boks{i}{j-1}$ is shrunk from the left and extended up, but
    since we choose $(i'',j'')$ as the rightmost point in $\mathcal{K}$ we can
    be sure that this box is empty.
\end{itemize}
The effect on the rest of the mesh is as follows:
\begin{itemize}
  \item Every box $\boks{\ell}{j}$ ($\ell \neq i-1,i$) is shrunk (from below)
    and any empty box is still empty. In particular every shaded box can remain
    shaded.
  \item Every box $\boks{\ell}{j-1}$ ($\ell \neq i-1,i$) is extended into the
    previous $\boks{\ell}{j}$ box. Requirement~\eqref{lem:shading:req3} in the
    lemma ensures that a shaded box remains empty.
  \item Every box $\boks{i}{\ell}$ ($\ell \neq j,j-1$) is shrunk (from the
    left) and any empty box is still empty. In particular every shaded box can
    remain shaded.
  \item Every box $\boks{i-1}{\ell}$ ($\ell \neq j,j-1$) is extended into the
    previous $\boks{i}{\ell}$ box. Requirement~\eqref{lem:shading:req4} in the
    lemma ensures that a shaded box remains empty. \qedhere
\end{itemize}
\end{proof}

We note that a related result was independently proved by
Tenner~\cite[Theorem 3.5']{Tenner} which determines when a mesh pattern is
coincident to the underlying classical pattern.

\begin{example}
By using Lemma~\ref{lem:shading} the following coincidence can be found. The
point that the arrow is pointing from is the point $(i,j)$ in the lemma.
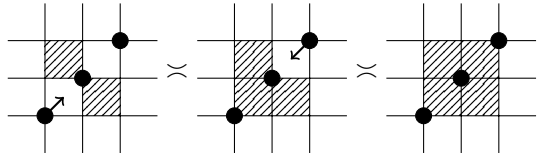
\[
\begin{tikzpicture}[scale=0.5, baseline=(current bounding box.center)]

    \draw (0.01,0.01) grid (3+0.99,3+0.99);

    \fill[pattern=north east lines] (1,2) rectangle +(1,1);
  \fill[pattern=north east lines] (2,1) rectangle +(1,1);

    \filldraw (1,1) circle (6pt);
    \filldraw (2,2) circle (6pt);
    \filldraw (3,3) circle (6pt);

  \draw [->,thick] (1,1) -- (1.5,1.5);

\end{tikzpicture}
 \asymp
\begin{tikzpicture}[scale=0.5, baseline=(current bounding box.center)]

    \draw (0.01,0.01) grid (3+0.99,3+0.99);

    \fill[pattern=north east lines] (1,2) rectangle +(1,1);
  \fill[pattern=north east lines] (2,1) rectangle +(1,1);
  \fill[pattern=north east lines] (1,1) rectangle +(1,1);

    \filldraw (1,1) circle (6pt);
    \filldraw (2,2) circle (6pt);
    \filldraw (3,3) circle (6pt);

  \draw [->,thick] (3,3) -- (2.5,2.5);

\end{tikzpicture}
 \asymp
\begin{tikzpicture}[scale=0.5, baseline=(current bounding box.center)]

    \draw (0.01,0.01) grid (3+0.99,3+0.99);

    \fill[pattern=north east lines] (1,1) rectangle +(1,1);
  \fill[pattern=north east lines] (2,2) rectangle +(1,1);
  \fill[pattern=north east lines] (2,1) rectangle +(1,1);
  \fill[pattern=north east lines] (1,2) rectangle +(1,1);

    \filldraw (1,1) circle (6pt);
    \filldraw (2,2) circle (6pt);
    \filldraw (3,3) circle (6pt);

\end{tikzpicture}
\]
First we can shade the box $\boks{1}{1}$ because when choosing the rightmost (or
the topmost point) in that box, neither of the shaded boxes, $\boks{1}{2}$ and
$\boks{2}{1}$, will be extended. When shading the box $\boks{2}{2}$ we choose
the leftmost (or the lowest point) and the same applies as before. Notice that
the point $(1,1)$ and the box $\boks{1}{0}$ do not fulfill the condition of the
lemma since the box $\boks{2}{1}$ is shaded but the box $\boks{2}{0}$ is not.
\end{example}
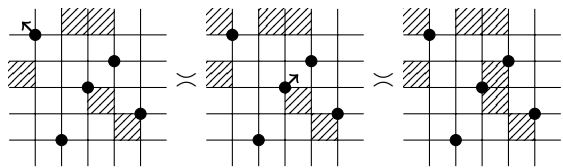
\begin{example}
\[
\begin{tikzpicture}[scale=0.35, baseline=(current bounding box.center)]

    \draw (0.01,0.01) grid (5+0.99,5+0.99);

  \fill[pattern=north east lines] (0,3) rectangle +(1,1);
    \fill[pattern=north east lines] (3,5) rectangle +(1,1);
    \fill[pattern=north east lines] (4,1) rectangle +(1,1);
    \fill[pattern=north east lines] (2,5) rectangle +(1,1);
    \fill[pattern=north east lines] (3,2) rectangle +(1,1);

    \filldraw (1,5) circle (6pt);
    \filldraw (2,1) circle (6pt);
    \filldraw (3,3) circle (6pt);
    \filldraw (4,4) circle (6pt);
    \filldraw (5,2) circle (6pt);

  \draw [->,thick] (1,5) -- (0.5,5.5);

\end{tikzpicture}
 \asymp
\begin{tikzpicture}[scale=0.35, baseline=(current bounding box.center)]

    \draw (0.01,0.01) grid (5+0.99,5+0.99);

  \fill[pattern=north east lines] (0,3) rectangle +(1,1);
    \fill[pattern=north east lines] (3,5) rectangle +(1,1);
    \fill[pattern=north east lines] (4,1) rectangle +(1,1);
    \fill[pattern=north east lines] (0,5) rectangle +(1,1);
    \fill[pattern=north east lines] (2,5) rectangle +(1,1);
    \fill[pattern=north east lines] (3,2) rectangle +(1,1);

    \filldraw (2,1) circle (6pt);
    \filldraw (3,3) circle (6pt);
    \filldraw (4,4) circle (6pt);
    \filldraw (5,2) circle (6pt);
  \filldraw (1,5) circle (6pt);

  \draw [->,thick] (3,3) -- (3.5,3.5);

\end{tikzpicture}
 \asymp
\begin{tikzpicture}[scale=0.35, baseline=(current bounding box.center)]

    \draw (0.01,0.01) grid (5+0.99,5+0.99);

  \fill[pattern=north east lines] (0,3) rectangle +(1,1);
    \fill[pattern=north east lines] (3,5) rectangle +(1,1);
    \fill[pattern=north east lines] (4,1) rectangle +(1,1);
    \fill[pattern=north east lines] (0,5) rectangle +(1,1);
    \fill[pattern=north east lines] (3,3) rectangle +(1,1);
    \fill[pattern=north east lines] (2,5) rectangle +(1,1);
    \fill[pattern=north east lines] (3,2) rectangle +(1,1);

    \filldraw (1,5) circle (6pt);
    \filldraw (2,1) circle (6pt);
  \filldraw (3,3) circle (6pt);
    \filldraw (4,4) circle (6pt);
    \filldraw (5,2) circle (6pt);

\end{tikzpicture}
\]
We can shade the box $\boks{0}{5}$ by choosing the leftmost (or the topmost
point) in the box, because none of the shaded boxes touching the lines $x=0$ or
$y=5$ will be extended by this. We can also shade the box $\boks{3}{3}$ by
choosing the rightmost (\emph{not} the topmost point!). This is because the
shaded boxes touching the lines $x=3$ and $y=3$ will either not be extended,
boxes $\boks{0}{3}$ and $\boks{3}{5}$, or are extended into a shaded area, box
$\boks{2}{5}$. However when trying to shade box $\boks{1}{0}$, box $\boks{4}{1}$
would have to be extended into a non-shaded area. Therefore box $\boks{1}{0}$
cannot be shaded.
\end{example}

We record one more new operation that preserves Wilf-equivalence in
Appendix~\ref{app:switch} since it is only useful for patterns of length three
or more.

\section{Wilf-classes}
\begin{definition}
A \emph{Wilf-subclass} is a set containing patterns of the same length which
are Wilf-equivalent. A \emph{Wilf-class} is a maximal Wilf-subclass.
\end{definition}

In this section we look at the Wilf-classification of mesh patterns of length
$1$ and $2$.
\newline

The number of mesh patterns of length $1$ is $2\cdot 2^3 = 16$. The operations
from above suffice to sort these $16$ patterns into $4$ Wilf-classes. Below is
one representative from each class.
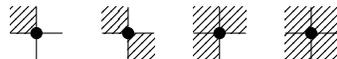
\[
\pattern{scale = 1}{1}{1/1}{0/1}  \quad \pattern{scale = 1}{1}{1/1}{0/1,1/0} \quad \pattern{scale = 1}{1}{1/1}{0/0,0/1,1/1} \quad \pattern{scale = 1}{1}{1/1}{0/0,0/1,1/0,1/1}
\]
Occurrences of these patterns in a permutation are well-known. Each occurrence
of the first pattern in a permutation $\pi$ is a \emph{left-to-right maximum}.
An occurrence of the second pattern in a permutation $\pi$ is a \emph{strong
fixed point} in $\pi$.

\begin{observation}[Stanley~\cite{MR1442260}] \label{obs:genstrongfix}
For reference we record the generating function for the number of permutations
without strong fixed points is
\[
G(x)= \frac{F(x)}{1+xF(x)} \quad \text{where } F(x) = \sum_{n\geq0} n!x^n.
\]
\end{observation}

\noindent
Both left-to-right maxima and strong fixed points in connection to mesh patterns
were introduced by Br\"and\'en and Claesson~\cite{MR2795782}. All permutations
of length $n$ that contain the third pattern begin with $n$. There is only one
permutation that contains the last pattern, namely $1$.
\newline

The number of mesh patterns of length $2$ is $2^{10}=1024$. We start by creating
$1024$ Wilf-subclasses, one for each of the patterns. By using the operations
above many of these subclasses merge. More precisely, the operations reverse,
complement and inverse preserve Wilf-equivalence so after merging we arrive at
$168$ Wilf-subclasses. When further using the shading lemma this number
decreases to $87$ and finally when taking into account the toric shift and the
up-shift we arrive at $65$ Wilf-subclasses.

Below we further reduce this number to $56$ by combining some of the above
classes. We conjecture that the actual number of Wilf-classes is $46$.

\subsection{The Wilf-class containing classical patterns}
Here we will bring the number of Wilf-subclasses down to $62$ by merging four
Wilf-subclasses, represented by the patterns shown in
Table~\ref{tab:classical1}.
\begin{table}[ht]
{
\renewcommand{\arraystretch}{1.6}
\begin{tabular}{c|c|l|c}
\multirow{2}{*}{Nr.\ } & \multirow{2}{*}{Repr.\ $p$}  & \multirow{2}{*}{$|S_n(p)|$ for $n = 1, \dotsc, 9$}  & size of \\[-1.75ex]
&&&subclass\\
\hline
\hline
1 & $\pattern{scale = 0.6}{2}{1/1,2/2}{}$ & $1,1,1,1,1,1,1,1,1$ & $98$\\%/1
2 & $\pattern{scale = 0.6}{2}{1/1,2/2}{0/1,1/2}$ & $1,1,1,1,1,1,1,1,1$ & $8$\\%/11
3 & $\pattern{scale = 0.6}{2}{1/1,2/2}{0/0,0/1,1/2}$ & $1,1,1,1,1,1,1,1,1$ & $16$\\%/47
4 & $\pattern{scale = 0.6}{2}{1/1,2/2}{0/1,1/2,0/0,2/2}$ & $1,1,1,1,1,1,1,1,1$ & $4$\\%/134
\end{tabular}
}
\caption{Patterns Wilf-equivalent to classical patterns}
\label{tab:classical1}
\end{table}

Because of Observation~\ref{obs:shading} it suffices to show that a permutation
containing the classical pattern $12$ also contains the last pattern in
Table~\ref{tab:classical1}. Consider a permutation containing $12$. Out of all
the occurrences of $12$ we can choose the occurrence $\pi_i \pi_j$ such that
there are no other occurrence $\pi_{i'}\pi_{j'}$ with $i'<i$. If there
are multiple occurrences $\pi_i\pi_{j_1}, \pi_i\pi_{j_2}, \dots, \pi_i\pi_{j_k}$,
we can choose the one where the value $\pi_{j_\ell}$ is maximized. This
particular occurrence, $\pi_i\pi_{j_\ell}$, is an occurrence of
$\pattern{scale = 0.5}{2}{1/1,2/2}{0/1,1/2,0/0,2/2}$.

There are no other mesh patterns of length $2$ that have the same number of
avoiding permutations of length $n = 1, \dotsc, 9$ as the patterns in
Table~\ref{tab:classical1}, so this Wilf-class is comprised of exactly the
Wilf-subclasses in the table. This result is shown in
Table~\ref{tab:classical2}. It is worth noticing that over $1/10$th of length
$2$ mesh patterns lie in this class.
\begin{table}[ht]
{
\renewcommand{\arraystretch}{1.6}
\begin{tabular}{c|c|c|c}
\multirow{2}{*}{Repr.\ $p$} & \multirow{2}{*}{$|S_n(p)|$} & size of & \multirow{2}{*}{OEIS seq.} \\[-1.75ex]
&& class & \\
\hline
\hline
$\pattern{scale = 0.6}{2}{1/1,2/2}{}$ & $1$ & $126$ & A000012 \\
\end{tabular}}
\caption{The Wilf-class containing classical patterns}
\label{tab:classical2}
\end{table}

\subsection{Wilf-classes containing vincular patterns}
In this section we will deal with subclasses that contain vincular patterns and
subclasses that can be merged with them. These subclasses are shown in
Table~\ref{tab:vincular1}.  We will bring the number of subclasses down to $58$
by combining the first five subclasses into one Wilf-class.

\begin{table}[ht]
{
\renewcommand{\arraystretch}{1.6}
\begin{tabular}{c|c|l|c}
\multirow{2}{*}{Nr.\ } & \multirow{2}{*}{Repr.\ $p$}  & \multirow{2}{*}{$|S_n(p)|$ for $n = 1, \dotsc, 9$}  & size of \\[-1.75ex]
&&&subclass\\
\hline
\hline
5 & $\pattern{scale = 0.6}{2}{1/1,2/2}{0/0,0/1,0/2}$ & $1,1,2,6,24,120,720,5040,40320$ & $184$\\ %/6
6 & $\pattern{scale = 0.6}{2}{1/1,2/2}{0/1,1/1,2/0,2/2}$ & $1,1,2,6,24,120,720,5040,40320$ & $8$\\ %/184
7 & $\pattern{scale = 0.6}{2}{1/1,2/2}{0/0,0/2,1/1}$ & $1,1,2,6,24,120,720,5040,40320$ & $16$\\ %/104
8 & $\pattern{scale = 0.6}{2}{1/1,2/2}{0/0,0/1,1/0,1/1}$ & $1,1,2,6,24,120,720,5040,40320$ & $8$\\ %/58
9 & $\pattern{scale = 0.6}{2}{1/1,2/2}{0/1,1/1,1/2,2/1}$ & $1,1,2,6,24,120,720,5040,40320$ & $16$\\[0.5ex] %/48
\hline
10 & $\pattern{scale = 0.6}{2}{1/1,2/2}{0/0,0/1,0/2,2/0,2/1,2/2}$ & $1,1,3,12,60,360,2520,20160,181440$ & $80$\\[0.5ex] %/65
\hline
11 & $\pattern{scale = 0.6}{2}{1/1,2/2}{0/0,0/1,0/2,1/0,1/1,1/2,2/0,2/1,2/2 %/512
}$ & $1,1,6,24,120,720,5040,40320, 362880$ & $2$\\
\end{tabular}
}
\caption{Patterns Wilf-equivalent to vincular patterns}
\label{tab:vincular1}
\end{table}

\begin{proposition}\label{prop:4} Here we consider the first three subclasses
from Table~\ref{tab:vincular1}.
\begin{enumerate}

  \item A permutation of length $n$ avoids the pattern
  \[
    \pattern{scale = 0.7}{2}{1/1,2/2}{0/0,0/1,0/2}
  \]
  if and only if it starts with $n$. Therefore, the number
  of permutations of length $n$ that avoid this pattern is
  $(n-1)!$.

  \item \label{prop:4part2} A permutation avoids the pattern
  \[
    \pattern{scale = 0.7}{2}{1/1,2/2}{0/1,1/1,2/0,2/2}
  \]
  if and only if it ends with $1$. Therefore, the number
  of permutations of length $n$ that avoid this pattern is
  $(n-1)!$.

  \item A permutation of length $n$ avoids the pattern
  \[
    \pattern{scale = 0.7}{2}{1/1,2/2}{0/0,0/2,1/1}
  \]
  if and only if it starts with $n$. Therefore, the number
  of permutations of length $n$ that avoid this pattern is
  $(n-1)!$.

\end{enumerate}
\end{proposition}

\begin{proof}
We only prove part \eqref{prop:4part2} as the others are similar. Let $\pi$ be a
permutation that does not end with $1$. Let $a$ be the last element of $\pi$ and
let $b = a-1$, which is somewhere to the left of $a$ in $\pi$. It is easy to see
that the letters $ab$ form an occurrence of the pattern. Clearly if $\pi$ ends
with $1$ it can not contain the pattern.
\end{proof}

Recall that $[x^n]f(x)$ denotes the coefficient of $x^n$ in the power series
expansion of the function $f(x)$.

\begin{definition}\label{def:Eulerian}
The $n$-th Eulerian polynomial $E_n(x)$ can be defined by
\[
\sum_{k=0}^{\infty} (k+1)^n x^k = \frac{E_n(x)}{(1-x)^{n+1}},
\]
see for example~\cite{MR2502493}. It is well-known that the Eulerian numbers,
$T_{n,k} = [x^k]E_n$, can also be defined recursively by
\[
T_{n, k} = k \cdot T_{n-1, k} + (n-k+1) \cdot T_{n-1, k-1}, \quad T_{1, 1} = 1.
\]
\end{definition}

Instead of just enumerating avoiders of the fourth and fifth pattern in
Table~\ref{tab:bivincular1} we find the distribution of descents and use it to
find the enumeration.
\begin{proposition}\label{prop:6}
The distribution for the number of descents in permutations of length $n$
avoiding the pattern $p =\pattern{scale=0.5}{2}{1/1,2/2}{0/0,0/1,1/0,1/1}$ is
\[
\sum_{\pi \in S_n(\pattern{scale=0.4}{2}{1/1,2/2}{0/0,0/1,1/0,1/1})} x^{\des(\pi)}=xE_{n-1}(x),
\]
where $\des(\pi)$ is the number of descents in the permutation $\pi$. This implies
\[
\left| S_n\left(\pattern{scale=0.6}{2}{1/1,2/2}{0/0,0/1,1/0,1/1}\right) \right| = (n-1)!
\]
for all $n\geq 1$.
\end{proposition}

\begin{proof}
In order to construct all permutations of length $n$ we can take each
permutation in $S_{n-1}$ and place $n$ in every possible position. Let $\pi '$
be a permutation obtained by adding the letter $n$ to $\pi$, where $\pi$ is a
permutation in $S_{n-1}$ which contains the pattern $p$. Then there exist $i <
j$ where the only letter lower than $\pi(j)$ in positions $1,2, \ldots j-1$ is
$\pi(i)$. From this it follows that we can add the letter $n$ to any position in
the permutation $\pi$ for $\pi'$ to still satisfy these conditions and therefore
contain the pattern $p$.

Now let $\pi$ be a permutation in $S_{n-1}$ that avoids $p$. Then for all $i <
j$ with $\pi(i) < \pi(j)$ there are at least two letters in places $1,2, \ldots
j-1$ lower than $\pi(j)$. In order to get a permutation of length $n$ that also
avoids $p$ we can add the letter $n$ in all positions except between the letters
$\pi(1)$ and $\pi(2)$. By adding the letter $n$ between the letters $\pi(1)$ and
$\pi(2)$ we produce an ascent where the letter $n$ has only one letter to the
left of it and therefore $\pi'$ would contain the pattern $p$. However by adding
the letter $n$ in all other positions does not produce an ascent or it produces
an ascent that has more than one letter to the left of it and therefore $\pi'$
avoids the pattern $p$. Since we are not allowed to add the letter $n$ between
the letters $\pi(1)$ and $\pi(2)$, it follows that all permutations in
$S_n(\pattern{scale=0.5}{2}{1/1,2/2}{0/0,0/1,1/0,1/1})$ start with a descent.

We let $B_n(x)$ be the distribution for the number of descents in permutations
of length $n$ avoiding $p$, i.e.,
\[
  B_n(x) = \sum_{\pi \in S_n(\pattern{scale=0.4}{2}{1/1,2/2}{0/0,0/1,1/0,1/1})} x^{\text{des} (\pi)}, \quad \text{and } R_{n,k} = [x^k]B_n(x).
\]
In order to construct a permutation $\pi$ of length $n$ with $k$ descents, we
have two choices. We can either take a permutation of length $n-1$ with $k$
descents and add $n$ to it without changing the number of descents or take a
permutation of length $n-1$ with $k-1$ descents and increase the number of
descents by one by adding $n$ to it. By adding $n$ between two letters making up
a descent we do not change the number of descents. If we add $n$ in any of the
other position we increase the number of descents by one. Above we have shown
that we can add the letter $n$ to a permutation of length $n-1$ that avoids $p$
in every position except between the first two letters of the permutation. We
have also shown that every permutation avoiding $p$ begins with a descent. From
this it follows that
\[
R_{n,k} = (k-1)R_{n-1,k} + (n-k+1)R_{n-1,k-1}.
\]
We claim that $B_n(x)=xE_{n-1}(x)$, or equivalently $R_{n,k}=T_{n-1,k-1}$. We
will prove this by induction. For $n=2$ we have $B_2(x)=x^1$ and $E_1(x)=x^0$,
so $R_{2,1}=1$ and $T_{1,0}=1$. Hence, the claim holds for the base case. Now we
assume that $R_{N,k}= T_{N-1, k-1}$ holds for all $N<n$ and all $k$. By
Definition~\ref{def:Eulerian}, we have
\begin{align*}
R_{n,k} &= (k-1)R_{n-1,k} + (n-k+1)R_{n-1,k-1}\\
&= (k-1)T_{n-2, k-1} + (n-k+1)T_{n-2,k-2} \tag{by induction hypothesis}\\
&= T_{n-1,k-1}.
\end{align*}
Thus, the claim holds for all $n$. It is well-known that
$\sum_{k=0}^{n-1}T_{n-1,k}=(n-1)!$ and therefore the number of permutations of
length $n$ avoiding p is $(n-1)!$.
\end{proof}

The proof of the next proposition is analogous to the previous one, but instead
of considering where $n$ can be added to a permutation we consider where $1$ can
be added (and the rest of the values raised by $1$).

\begin{proposition}\label{prop:7}
The distribution for the number of descents in permutations of length $n$
avoiding the pattern $p =\pattern{scale=0.5}{2}{1/1,2/2}{0/1,1/2,2/1,1/1}$ is
\[
\sum_{\pi \in S_n(\pattern{scale=0.4}{2}{1/1,2/2}{0/1,1/1,1/2,2/1})} x^{\des (\pi)}=xE_{n-1}(x).
\]
This implies
\[
\left| S_n\left(\pattern{scale=0.6}{2}{1/1,2/2}{0/1,1/1,1/2,2/1}\right) \right| = (n-1)!
\]
for all $n\geq 0$.
\end{proposition}

% More detailed proof of prop:7: To construct all permutations of length $n$ we take each permutation $\pi$ in $S_{n-1}$ and place the letter $1$ in every position in $\pi$ and add $1$ to the remaining letters. Now let us take a permutation $\pi$ in $S_{n-1}$ which contains $p$, that is, there exist $1\leq i<j\leq n$ such that there does not exist a letter $k$ such that $\pi(i)<k<\pi (j)$, and we do not have a letter $\ell$ such that $\ell > \pi(i)$. We can place $1$ in any position in $\pi$ for these conditions to be satisfied for the new permutation, $\pi '$.

% Now let us take a permutation $\pi$ in $S_{n-1}$ which avoids $p$, that is, for all $1\leq i<j\leq n$ with $\pi(i)<\pi(j)$ we have a letter $k$ where $\pi(i)<k<\pi (j)$, or a letter $\ell$ such that $\ell > \pi(i)$. If we place  $1$ in front of $2$ in $\pi$, $\pi'$ contains an ascent with the letters $12$, and therefore clearly contains the pattern. If $1$ is placed in every other position in $\pi$, a non-inversion is created by the letter $1$ and any letter greater than $2$. Thus the conditions will be satisfied in $\pi '$ and $\pi'$ will avoid the pattern. By similar arguments as in Proposition~\ref{prop:6}, we find this pattern gives the same distribution for number of descents as the pattern in the previous proposition. Hence the number of permutations of length $n$ avoiding the pattern $p$ is $(n-1)!$.

This takes care of merging the first five Wilf-subclasses into one Wilf-class.
The remaining propositions in this section give the enumeration of the remaining
patterns. Their proofs are left for the reader.

% Pattern class 65
\begin{proposition}
  A permutation contains the pattern
  \[
    \pattern{scale = 0.7}{2}{1/1,2/2}{0/0,0/1,0/2,2/0,2/1,2/2}
  \]
  if and only if its first letter is smaller than its last.
  This happens for precisely half of the permutations. Thus
  the the number of permutations of length $n$ avoiding the pattern is
  $n!/2$ for $n \geq 2$.
\end{proposition}

% Pattern class 512
\begin{proposition}
  The only permutation that contains the pattern
  \[
    \pattern{scale = 0.7}{2}{1/1,2/2}{0/0,0/1,0/2,1/0,1/1,1/2,2/0,2/1,2/2}
  \]
  is $12$. Therefore the number of permutations of length $n$
  that avoid the pattern is $n!$ for $n \geq 3$.
\end{proposition}

Propositions~\ref{prop:4}, \ref{prop:6} and~\ref{prop:7} show that the first
five subclasses of Table~\ref{tab:vincular1} have the same enumeration.

\begin{table}[hb]
{
\renewcommand{\arraystretch}{1.6}
\begin{tabular}{c|c|c|c}
\multirow{2}{*}{Repr.\ $p$} & \multirow{2}{*}{$|S_n(p)|$} & size of & \multirow{2}{*}{OEIS seq.} \\[-1.75ex]
&& class & \\
\hline
\hline
$\pattern{scale = 0.6}{2}{1/1,2/2}{0/0,0/1,0/2}$ & $(n-1)!$ & $232$ & A000142 \\[0.5ex]
\hline
$\pattern{scale = 0.6}{2}{1/1,2/2}{0/0,0/1,0/2,2/0,2/1,2/2}$ & $n!/2$, $n \geq 2$ & $80$ & A001710 \\[0.5ex]
\hline
$\pattern{scale = 0.6}{2}{1/1,2/2}{0/0,0/1,0/2,1/0,1/1,1/2,2/0,2/1,2/2}$ & $n!$, $n \geq 3$ & $2$ & A000142 ($n \geq 3$)\\
\end{tabular}}
\caption{Wilf-classes containing vincular patterns}
\label{tab:vincular2}
\end{table}

\subsection{Wilf-classes containing bivincular patterns}
% Pattern classes 281, 303, 503 containing a bivincular pattern
% Pattern class 67 has the same number sequence as 281
Here we will deal with Wilf-subclasses that contain bivincular patterns and
other subclasses that can be merged with them. The subclasses are shown in
Table~\ref{tab:bivincular1}. We will bring the number of subclasses down to $57$
by combining the last two subclasses.

\begin{table}[ht]
{
\renewcommand{\arraystretch}{1.6}
\begin{tabular}{c|c|l|c}
\multirow{2}{*}{Nr.\ } & \multirow{2}{*}{Repr.\ $p$}  & \multirow{2}{*}{$|S_n(p)|$ for $n = 1, \dotsc, 9$}  & size of \\[-1.75ex]
&&&subclass\\
\hline
\hline
12 & $\pattern{scale = 0.6}{2}{1/1,2/2}{0/0,0/1,0/2,1/0,2/0}$ & $1,1,4,18,96,600,4320,35280, 322560$ & $56$\\[0.5ex] %/303
\hline
13 & $\pattern{scale = 0.6}{2}{1/1,2/2}{0/0,0/1,0/2,1/0,1/2,2/0,2/1,2/2}$ & $1, 1, 5, 22, 114, 696, 4920, 39600, 357840$ & $18$\\[0.5ex] %/503
\hline
14 & $\pattern{scale = 0.6}{2}{1/1,2/2}{0/1,1/1,1/2,1/0,1/2,2/1}$ & $1, 1, 3, 11, 53, 309, 2119, 16687, 148329$ & $2$\\ %/281
15 & $\pattern{scale = 0.6}{2}{1/1,2/2}{0/1,0/2,1/0,1/1,1/2}$ & $1, 1, 3, 11, 53, 309, 2119, 16687, 148329$ & $32$\\ %/67
\end{tabular}
}
\caption{Patterns Wilf-equivalent to bivincular patterns}
\label{tab:bivincular1}
\end{table}

Formulas for the number of permutations avoiding patterns in the first three
subclasses were found by Parviainen~\cite{RB}. Computer experiments show that
these subclasses can not be enlarged further. He also shows that the number of
permutations of length $n$ that avoid patterns in the third subclass is
\[
\sum_{k=0}^n (-1)^k (n-k+1) \frac{n!}{k!}.
\]
The following proposition shows that the last subclass in the table can be
merged with the third one.

\begin{proposition}\label{prop:10}
A permutation $\pi$ avoids the pattern
$p = \pattern{scale = 0.5}{2}{1/1,2/2}{0/1,0/2,1/0,1/1,1/2}$ if and only if each
ascent in $\pi$ is the $12$ of a $312$ pattern or the $13$ of a $213$ pattern,
or both. If
\[
\left| S_n\left(\pattern{scale = 0.6}{2}{1/1,2/2}{0/1,0/2,1/0,1/1,1/2}\right) \right| = a_n
\]
then $a_n = (n-1)a_{n-1} +(n-2)a_{n-2}$ and $a_0 = a_1 = 1$.
Furthermore
\[
a_n = \sum_{k=0}^n (-1)^k (n-k+1) \frac{n!}{k!}.
\]
\end{proposition}

\begin{proof} We leave the characterization in terms of ascents to the reader,
as well as the last formula for $a_n$. Let $A_n$ be the set of all permutations
that avoid $p$ and $a_n$ be the size of $A_n$. Now, let $A_{n,k}=\{\pi\in
A(n) \colon \pi(n)=k\}$. For $k\neq n$ we define a mapping $\varphi_k:A_{n,k}\to
A_{n-1}$ such that for a permutation $\pi$, the mapping removes $k$ from $\pi$
and then subtracts $1$ from all letters larger than $k$. Then we can also define
$\varphi_k^{-1}: A_{n-1}\to A_{n,k}$ to be the mapping that appends $k$ to the
end of a permutation $\pi$ and then adds $1$ to all letters that are equal to
$k$ or larger. For $k=n$ we also have a mapping $\varphi_n: A_{n,n} \to
A_{n-1}\setminus A_{n-1,n-1}$ such that for a permutation $\pi$ where the last
letter is $n$, the mapping removes $n$. The range of the mapping is
$A_{n-1}\setminus A_{n-1,n-1}$ because if $\varphi(\pi)(n)=n-1$ then $\pi$ would
end with $(n-1)n$ which is an occurrence of the pattern $p$. It is easy to see
that the inverse is $\varphi_{n}^{-1}= A_{n-1}\setminus A_{n-1,n-1}\to A_{n,n}$;
the mapping that appends $n$ to the end of a permutation.

As explained above, in order to construct all permutations of length $n$
avoiding the pattern $A_n$, we can append $n$ at the end of all permutations of
length $n-1$ except for those ending with the letter $n-1$. Hence,
\[
a_n = na_{n-1} - a_{n-1,n-1}.
\]
From the mappings we get $a_{n,n} = a_{n-1} - a_{n-1,n-1}$ which gives
$a_{n-1,n-1} = a_{n-1} - a_{n,n}$. We also have $a_{n,n} = (n-2)a_{n-2}$ since
for producing a permutation from $A_{n,n}$ we take a permutation in $A_{n-2}$
and append a letter from the set $\{1,2,\ldots,n-2\}$ to it. Lastly we append
$n$ at the end of the permutation. Therefore,
\begin{align*}
a_n &= na_{n-1} - (a_{n-1} - a_{n,n})\\
&= na_{n-1} - a_{n-1} + (n-2)a_{n-2}\\
&= (n-1)a_{n-1} + (n-2)a_{n-2},
\end{align*}
which is what we wanted to prove.
\end{proof}

The results for the bivincular Wilf-classes are shown in
Table~\ref{tab:bivincular2}.

\begin{table}[ht]
{
\renewcommand{\arraystretch}{1.6}
\begin{tabular}{c|c|c|c}
\multirow{2}{*}{Repr.\ $p$} & \multirow{2}{*}{$|S_n(p)|$} & size of & \multirow{2}{*}{OEIS seq.} \\[-1.75ex]
&& class & \\
\hline
\hline
$\pattern{scale=0.6}{2}{1/1,2/2}{0/0,0/1,0/2,1/0,2/0}$             & $n! - (n-1)!$ ($n \geq 2$)                  & $56$ & A094258 \\[0.5ex]
\hline
$\pattern{scale=0.6}{2}{1/1,2/2}{0/0,0/1,0/2,1/0,1/2,2/0,2/1,2/2}$ & $n! - (n-2)!$ ($n \geq 2$)                  & $18$ & A213167 ($n \geq 2$) \\[0.5ex]
\hline
$\pattern{scale=0.6}{2}{1/1,2/2}{0/1,1/1,1/2,1/0,1/2,2/1}$         & $\sum_{k=0}^n (-1)^k (n-k+1) \frac{n!}{k!}$ & $34$ & A000255\\
\end{tabular}}
\caption{Wilf-classes containing bivincular patterns}
\label{tab:bivincular2}
\end{table}

\subsection{Enumerated Wilf-classes, not containing bivincular patterns}
In this section we provide formulas for the enumeration of the Wilf-subclasses
shown in Table~\ref{tab:notbivincwithseq1}. The first seven subclasses have
unique enumeration sequences for $n\leq 9$ and can therefore not be enlarged
further. The last two subclasses are shown to have the same enumeration in
Propositions~\ref{prop:8} and~\ref{prop:9}. They therefore merge into one
Wilf-class. Hence, the number of Wilf-classes is decreased from $57$ to $56$.
\begin{table}[ht]
{
\renewcommand{\arraystretch}{1.6}
\begin{tabular}{c|c|l|c}
\multirow{2}{*}{Nr.\ } & \multirow{2}{*}{Repr.\ $p$}  & \multirow{2}{*}{$|S_n(p)|$ for $n = 1, \dotsc, 9$}  & size of \\[-1.75ex]
&&&subclass\\
\hline
\hline
16 & $\pattern{scale = 0.6}{2}{1/1,2/2}{0/1,2/0,1/0,0/2}$ & $1, 1, 4, 17, 91, 574, 4173, 34353, 316012$ & $60$\\[0.5ex] %/178
\hline
17 & $\pattern{scale = 0.6}{2}{1/1,2/2}{0/1,1/2,0/0,2/0,1/0,0/2,2/1}$ & $1, 1, 5, 21, 110, 677, 4817, 38956, 353237$ & $8$\\[0.5ex] %/468
\hline
18 & $\pattern{scale = 0.6}{2}{1/1,2/2}{0/0,0/1,0/2,1/2,2/0,2/2}$ & $1, 1, 3, 13, 70, 446, 3276, 27252, 253296$ & $84$\\[0.5ex] %/142
\hline
19 & $\pattern{scale = 0.6}{2}{1/1,2/2}{0/1,0/2,1/1,1/2,2/0,2/2}$ & $1, 1, 3, 14, 80, 528, 3948, 33072, 307584$ & $16$\\[0.5ex] %/387
\hline
20 & $\pattern{scale = 0.6}{2}{1/1,2/2}{0/0,0/1,0/2,1/1,1/2,2/0,2/1}$ & $1, 1, 4, 19, 104, 656, 4728, 38508, 350592$ & $24$\\[0.5ex] %/469
\hline
21 & $\pattern{scale = 0.6}{2}{1/1,2/2}{0/1,1/2,0/0,2/0,2/2}$ & $1, 1, 2, 8, 47, 332, 2644, 23296, 225336$ & $4$\\[0.5ex] %/266
\hline
22 & $\pattern{scale = 0.6}{2}{1/1,2/2}{0/1,1/2,0/0,2/0,2/2,1/1}$ & $1, 1, 3, 15, 89, 594, 4434, 36892, 340308$ & $4$\\[0.5ex] %/400
\hline
23 & $\pattern{scale = 0.6}{2}{1/1,2/2}{0/0,0/2,1/0,1/1,1/2}$ & $1, 1, 2, 7, 33, 191, 1304, 10241, 90865$ & $16$\\ %/163
24 & $\pattern{scale = 0.6}{2}{1/1,2/2}{0/0,0/1,1/0,1/1,1/2}$ & $1, 1, 2, 7, 33, 191, 1304, 10241, 90865$ & $16$\\ %/135
\end{tabular}
}
\caption{Wilf-subclasses not containing bivincular patterns, with a known
counting sequence}
\label{tab:notbivincwithseq1}
\end{table}

\begin{proposition}\label{prop:12}
The number of permutations of length $n$ that avoid the pattern
$p = \pattern{scale = 0.5}{2}{1/1,2/2}{0/1,2/0,1/0,0/2}$ is
\[
n \left[ x^n \right] \log\left(1+\sum_{n\geq 1}(n-1)!x^n\right).
\]
\end{proposition}

\begin{proof}
We define two sequences,
\[
a_n= \left| S_n\left(\pattern{scale = 0.6}{2}{1/1,2/2}{0/1,2/0,1/0,0/2}\right) \right| \quad \text{and} \quad b_n= \left| S_n\left(\pattern{scale = 0.6}{1}{1/1}{0/1,1/0}\right) \right|,
\]
and let us show that
\begin{equation}\label{proof:12}
a_n=b_n+b_{n-1}.
\end{equation}
We define a map,
\[
\varphi:
S_n\left(\pattern{scale = 0.6}{1}{1/1}{0/1,1/0}\right) \cup
S_{n-1}\left(\pattern{scale = 0.6}{1}{1/1}{0/1,1/0}\right) \longrightarrow
S_n(p).
\]
The mapping $\varphi$ maps $\pi\in S_n(\pattern{scale = 0.5}{1}{1/1}{0/1,1/0})$
to itself in $S_n(p)$, we know that $\pi$ avoids
$\pattern{scale = 0.5}{1}{1/1}{0/1,1/0}$, and therefore it also avoids $p$. For
$\pi\in S_{n-1}(\pattern{scale = 0.5}{1}{1/1}{0/1,1/0})$ the mapping $\varphi$
appends the letter $n$ after the last letter in $\pi$, and we obtain a
permutation of length $n$ with the letter $n$ in the $n$-th position.

The inverse map
\[
\varphi^{-1}:
S_n(p) \longrightarrow
S_n\left(\pattern{scale = 0.6}{1}{1/1}{0/1,1/0}\right) \cup
S_{n-1}\left(\pattern{scale = 0.6}{1}{1/1}{0/1,1/0}\right)
\]
can be defined as follows:
For $\pi\in S_n(p)$ that ends with the letter $n$, we remove $n$ and then
$\varphi^{-1}(\pi)$ is in $S_{n-1}(\pattern{scale = 0.5}{1}{1/1}{0/1,1/0})$.
For $\pi\in S_n(p)$ that does not end with the letter $n$, $\pi$ maps to itself
in $S_n(\pattern{scale = 0.5}{1}{1/1}{0/1,1/0})$. The permutation $\pi$ avoids
$\pattern{scale = 0.5}{1}{1/1}{0/1,1/0}$ as well for the following reasons. For
a letter $u$ in $\pi$ that appears to the left of the letter $n$ there must be
points in at least one of the shaded areas for $\pi$ to avoid $p$. Also for $v$
in $\pi$ that appears to the right of $n$ then there is at least one point, $n$,
in the upper shaded box. Therefore, $\pi$ avoids
$\pattern{scale = 0.5}{1}{1/1}{0/1,1/0}$, and hence
$\varphi^{-1}(\pi)\in S_n(\pattern{scale = 0.5}{1}{1/1}{0/1,1/0})$. This proves
equation~\ref{proof:12}.

We define the generating function for $a_n$ to be
\[
D(x)=\sum_{n\geq 1}a_nx^n.
\]
The generating function for $b_n$ is given in
Observation~\ref{obs:genstrongfix}. Since $a_n=b_n+b_{n-1}$ we obtain
\begin{align*}
D(x)&=G(x)+xG(x)-1\\
&=\frac{\sum_{n\geq0} n!x^n}{1+x\sum_{n\geq 0}n!x^n}+x\left(\frac{\sum_{n\geq0} n!x^n}{1+x\sum_{n\geq 0}n!x^n}\right)-1\\
&=\frac{\sum_{n\geq 0}n!x^n-1}{1+x\sum_{n\geq 0}n!x^n}=\frac{\sum_{n\geq 1}n!x^n}{1+\sum_{n\geq 1}(n-1)!x^n}.
\end{align*}
We now have $D(x)=xA'(x)$ where $A(x)=\log(1+\sum_{n\geq 1}(n-1)!x^n)$ which
implies that $A(x)$ is the logarithmic generating function of $a_n$, i.e.,
\[
A(x)=\sum_{n\geq 1}\frac{a_n}{n}x^n=\log(1+\sum_{n\geq 1}(n-1)!x^n). \qedhere
\]
\end{proof}

\begin{proposition}\label{prop:18}
The number of permutations of length $n$ which have $n$ in position $i$ counted
from the right and contain the pattern
$p = \pattern{scale = 0.5}{2}{1/1,2/2}{0/0,0/1,0/2,1/2,2/0,2/2}$ is
$\frac{(n-1)!}{i}$ and therefore
\[
\left|S_n\left(\pattern{scale = 0.6}{2}{1/1,2/2}{0/0,0/1,0/2,1/2,2/0,2/2} \right)\right| = n! - \sum_{i=1}^{n-1}\frac{(n-1)!}{i}.
\]
\end{proposition}

\begin{proof}
To find an occurrence of the pattern $p$ in a permutation of length $n$, we must
use the first letter in the permutation and the letter $n$. Then it must also
hold that all the letters to the right of $n$ must be greater than the first
letter. Thus, the number of permutations containing the pattern $p$ depends on
the position of the letter $n$.

Recall that $i$ is the position of the letter $n$ counted from the right and let
$k$ be the size of the first letter. Obviously, the letter $n$ cannot be in the
last position counted from the right, and thus, $1\leq k \leq n-1$ and $1\leq i
\leq n-1$.

Now, we choose $i-1$ letters greater than $k$ to fill the positions to the right
of $n$, which can be done in $\binom{n-(k+1)}{i-1}$ ways. Then these letters can
be arranged in $(i-1)!$ ways. The remaining $n-i-1$ letters will be placed
between $n$ and $k$, which can be done in $(n-i-1)!$ ways. This must hold for
each $1\leq k \leq n-1$. Therefore, the number of permutations containing the
pattern $p$ is
\[
\sum_{k=1}^{n-1}\binom{n-k-1}{i-1}(i-1)! (n-i-1)!.
\]
We have
\begin{align*}
i\sum_{k=1}^{n-1}&\binom{n-k-1}{i-1}(i-1)! (n-i-1)!\\
&= i! (n-i-1)!\sum_{k=1}^{n-1}\binom{n-k-1}{i-1}\\
&= i! (n-i-1)!\binom{n-1}{i}\\
&= i! (n-i-1)!\frac{(n-1)!}{i!(n-1-i)!}\\
&=(n-1)!,
\end{align*}
which implies
\[
\sum_{k=1}^{n-1}\binom{n-k-1}{i-1}(i-1)! (n-i-1)! = \frac{(n-1)!}{i}. \qedhere
\]
\end{proof}

We note that a permutation contains the pattern in the next proposition if and only if it starts with
$1$ and the remaining letters form a permutation with a strong fixed point. The rest of the proof is left
to the reader
\begin{proposition}\label{prop:13}
A permutation $\pi$ of length $n$ contains the pattern
$p = \pattern{scale = 0.5}{2}{1/1,2/2}{0/1,1/2,0/0,2/0,1/0,0/2,2/1}$ if and only
if $\pi$ begins with $1$ and ends with a permutation of length $n-1$ that
contains a strong fixed point. Thus,
\[
\left|S_n \left(\pattern{scale = 0.6}{2}{1/1,2/2}{0/1,1/2,0/0,2/0,1/0,0/2,2/1} \right) \right| = n! - (n-1)! +  [x^{n-1}] \frac{F(x)}{1 + x F(x)},
\]
where $F(x) = \sum_{n\geq 0} n! x^n$, as in Observation~\ref{obs:genstrongfix}.
\end{proposition}
% A MORE DETAILED PROOF
% Let $q$ be the pattern $\pattern{scale = 0.5}{1}{1/1}{0/1,1/0}$. A permutation $\pi$ contains $q$ if and only if it contains a strong fixed point. Then, by appending the letter $1$ in front of each permutation of length $n-1$ containing the pattern $q$, and adding $1$ to the remaining letters, we clearly obtain all permutations containing the pattern $p$.

% Now, we know that the number of permutations of length $n-1$ that have no strong fixed points is\footnote{http://oeis.org/A006932}
% \[
% [x^{n-1}]\frac{F(x)}{1 + x F(x)}.
% \]
% Hence, the number of permutations of length $n-1$ that contain a strong fixed point is
% \[
% (n-1)! - [x^{n-1}]\frac{F(x)}{1 + xF(x)},
% \]
% which is then equal to the number of permutations of length $n$ containing the pattern $p$.

The proof of the next four propositions are similar to the proof of
Proposition~\ref{prop:18} and are therefore omitted.
\begin{proposition}\label{prop:15}
The number of permutations of length $n$ which have $n$ in position $i$, where
$i=0$ is the rightmost position and contain the pattern
$p = \pattern{scale = 0.5}{2}{1/1,2/2}{0/1,0/2,1/1,1/2,2/0,2/2}$ is
$i!(n-1-i)!$. Therefore
\[
\left|S_n\left(\pattern{scale = 0.6}{2}{1/1,2/2}{0/1,0/2,1/1,1/2,2/0,2/2} \right)\right| = n! - \sum_{i=0}^{n-2}i!(n-i-1)!.
\]
\end{proposition}
% A MORE DETAILED PROOF
% Let $n$ be in position $i$ where $i=0$ is the rightmost position. Thus, $n$ is the latter point in the pattern $p$. For a permutation $\pi$ to contain $p$, the letters $n-1,n-2, \ldots , n-i$ must be placed to the right of the letter $n$, that is, in the rightmost non-shaded box. The remaining $n-i-1$ letters are placed in the other non-shaded area. Thus, the number of permutations that have $n$ in position $i$ and contain the pattern $p$ is
% \[
% i!(n-1-i)!. \qedhere
% \]

\begin{proposition}\label{prop:14}
The number of permutations of length $n$ that start with $k$ and contain the
pattern $p = \pattern{scale = 0.5}{2}{1/1,2/2}{0/0,0/1,0/2,1/1,1/2,2/0,2/1}$ is
$(k-1)!(n-k-1)!$ and therefore
\[
\left|S_n\left(\pattern{scale = 0.6}{2}{1/1,2/2}{0/0,0/1,0/2,1/1,1/2,2/0,2/1} \right)\right| = n! - \sum_{k=1}^{n-1}(k-1)!(n-k-1)!.
\]
\end{proposition}
% A MORE DETAILED PROOF
% Let $k$ be the first letter in a permutation $\pi$. Thus, $k$ is the former point in the pattern $p$. For $\pi$ to contain the pattern $p$, the latter point of $p$ must be $k+1$. Also, the remaining letters of $\pi$ must be in the two non-shaded boxes in $p$. The letters from $1$ to $k-1$ must be placed in the lower non-shaded box, and the letters from $k+2$ to $n$ must be placed in the upper one. Hence, the number of letters in the lower non-shaded box is $k-1$ and the number of letters in the upper box is $n-(k+1)$. It follows that the number of permutations that start with $k$ and contain the pattern $p$ is
% \[
% (k-1)!(n-k-1)!. \qedhere
% \]

\begin{proposition}\label{prop:16}
The number of permutations of length $n$ containing the pattern $p$, with $i$ as
the height of the first point of the pattern
$p = \pattern{scale = 0.5}{2}{1/1,2/2}{0/1,1/2,0/0,2/0,2/2}$, counted from
above, and $\ell$ as the distance between the two points, is
$(i - \ell )! (n-i-\ell)! \ell !$.

Hence,
\[
\left| S_n\left( \pattern{scale = 0.6}{2}{1/1,2/2}{0/1,1/2,0/0,2/0,2/2} \right) \right| = n! - \sum_{i=1}^{n-1} \sum_{\ell =1 }^{i} (i - \ell )! (n-i-\ell)! \ell !,
\]
for all $n\geq 0$.
\end{proposition}
% A MORE DETAILED PROOF
% After some consideration, one can see that a permutation can either contain exactly one occurrence of the pattern $p$ or none at all.

% To construct a permutation containing the pattern, the following conditions must be satisfied. First of all we find the two letters in the permutation corresponding to the two points in the pattern, call them $a$ and $b$, respectively. We let $i$ be the height of the letter $a$ counted from above and $\ell$ be the distance between $a$ and $b$. Then, for a permutation of length $n$, we have $a =n-i$ and $b=n-i+\ell$.

% Second of all, we see that the letters greater than $b$ must be placed to the left of $a$. This can be done in $(i - \ell)!$ ways. Furthermore, all the $n-(i+1)=n-i-1$ letters lower than $a$ must be placed between $a$ and $b$, and they can be ordered in $(n-i-1)!$ ways. At last, the letters between $a$ and $b$ in size can be placed either between $a$ and $b$ or to the right of $b$. Thus, the number of ways to place these $n-i-\ell - (n-i)-1 = \ell -1$ letters is
% \begin{align*}
% \sum_{j=0}^{\ell -1} \binom{\ell -1 }{j} j! (\ell -1 -j) &= \sum_{j=0}^{\ell -1} (\ell -1)! = (\ell -1)!\sum_{j=0}^{\ell -1}1 = \ell !.\qedhere
% \end{align*}

\begin{proposition}\label{prop:17}
The number of permutations of length $n$ containing the pattern
$p = \pattern{scale = 0.5}{2}{1/1,2/2}{0/1,1/2,0/0,2/0,2/2,1/1}$ with $k$ as the
height of the first point of the pattern, counted from above, and $j$ the
distance between the two points is
\[
j!(k-j)!(n-k)!.
\]
Thus,
\[
\left| S_n\left (\pattern{scale = 0.6}{2}{1/1,2/2}{0/1,1/2,0/0,2/0,2/2,1/1}\right) \right| = n! - \sum_{k=0}^{n-2} \sum_{j=0}^{k} j!(k-j)!(n-2-k)!
\]
for all $n\geq 0$.
\end{proposition}
% A MORE DETAILED PROOF
% A permutation can either contain exactly one occurrence of the pattern $p$ or none at all.

% To construct a permutation containing the pattern, the following conditions must be satisfied. First of all we find the letters corresponding to the two points in the pattern $p$, let us call them $a$ and $b$, respectively. Let $k$ be the height of the letter $a$ counted from above and $j$ be the distance between $a$ and $b$. Then, in a permutation of length $n$, we have $a =n-k$ and $b=n-k+j$.

% Second of all, the letters greater than $b$ must be to the left of $a$. These $n-(n-k+j)= k - j$ letters can be arranged in $(k-j)!$ different ways. In addition, the $n-k-1$ letters lower than $a$ must be placed between $a$ and $b$, which can be done in $(n-k-1)!$ different ways. At last, the $n-k+j - (n-k)-1=j -1$ letters between $a$ and $b$ in size must be placed to the right of $b$, which can be done in $(j -1)!$ ways. Hence, the number of permutations containing the pattern $p$ is
% \[
% \sum_{k=1}^{n-1} \sum_{j = 1 }^{k}(n-k-1)!(k-j)!(j - 1)!= \sum_{k=0}^{n-2} \sum_{j=0}^{k} j!(k-j)!(n-2-k)!.\qedhere
% \]

% Only Wilf-class in this subsection that comes from
% two subclasses

The proofs of the next two proposition follow similar arguments as the proof of
Proposition~\ref{prop:10} and are therefore omitted.
\begin{proposition}\label{prop:8}
A permutation $\pi$ avoids the pattern
$p = \pattern{scale = 0.5}{2}{1/1,2/2}{0/0,0/2,1/0,1/1,1/2}$ if and only if each
ascent in $\pi$ is the $23$ of a $123$ pattern or the $12$ of a $312$ pattern,
or both. If,
\[
\left| S_n\left(\pattern{scale = 0.6}{2}{1/1,2/2}{0/0,0/2,1/0,1/1,1/2}\right) \right| = a_n
\]
then $a_n= n\cdot a_{n-1}-a_{n-2}$ and $a_{-1}=0, a_{0}=1$.
\end{proposition}

\begin{proposition}\label{prop:9}
A permutation $\pi$ avoids the pattern
$p = \pattern{scale = 0.5}{2}{1/1,2/2}{0/0,0/1,1/0,1/1,1/2}$ if and only if each
ascent in $\pi$ is the $13$ of a $213$ pattern or the $23$ of a $123$ pattern,
or both. If
\[
\left| S_n\left(\pattern{scale = 0.6}{2}{1/1,2/2}{0/0,0/1,1/0,1/1,1/2}\right) \right| = a_n
\]
then $a_n= n\cdot a_{n-1}-a_{n-2}$ and $a_{-1}=0, a_{0}=1$.
\end{proposition}
% A MORE DETAILED PROOF
% This statement can be proved using similar arguments as in
% Proposition~\ref{prop:8}, except here $k \in \dbrac{1,n} \setminus \{n-1\}$.

Table~\ref{tab:notbivincwithseq2} collects together the results of this section.
\begin{table}[ht]
{
\renewcommand{\arraystretch}{1.6}
\begin{tabular}{c|c|c|c}
\multirow{2}{*}{Repr.\ $p$} & \multirow{2}{*}{$|S_n(p)|$} & size of & \multirow{2}{*}{OEIS seq.} \\[-1.75ex]
&& class & \\
\hline
\hline
$\pattern{scale = 0.6}{2}{1/1,2/2}{0/1,2/0,1/0,0/2}$ & $n[x^n]\left( 1 + \sum_{i=1}^{n} (i-1)!\cdot x^i \right)$ & $60$ & A141154 \\[0.5ex]
\hline
$\pattern{scale = 0.6}{2}{1/1,2/2}{0/1,1/2,0/0,2/0,1/0,0/2,2/1}$ & $n! - (n-1)! + [x^n] \frac{F(x)}{1 + x F(x)}$ & $8$ & \\[0.5ex]
\hline
$\pattern{scale = 0.6}{2}{1/1,2/2}{0/0,0/1,0/2,1/2,2/0,2/2}$ & $n! - \sum_{i=1}^{n-1} \frac{(n-1)!}{i}$ & $84$ & A121586 ($n \geq 2$)\\[0.5ex]
\hline
$\pattern{scale = 0.6}{2}{1/1,2/2}{0/1,0/2,1/1,1/2,2/0,2/2}$ & $n! - \sum_{i=0}^{n-2}i!(n-1-i)!$ & $16$ \\[0.5ex]
\hline
$\pattern{scale = 0.6}{2}{1/1,2/2}{0/0,0/1,0/2,1/1,1/2,2/0,2/1}$ & $n! - \sum_{k=1}^{n-1}(k-1)!(n-k-1)!$ & $24$ \\[0.5ex]
\hline
$\pattern{scale = 0.6}{2}{1/1,2/2}{0/1,1/2,0/0,2/0,2/2}$ & $n! - \sum_{i=1}^{n-1} \sum_{\ell =1 }^{i} (i - \ell )! (n-i-\ell)! \ell !$ & $4$ \\[0.5ex]
\hline
$\pattern{scale = 0.6}{2}{1/1,2/2}{0/1,1/2,0/0,2/0,2/2,1/1}$ & $n! -  \sum_{k=0}^{n-2} \sum_{j=0}^{k} j!(k-j)!(n-2-k)!$ & $4$ \\[0.5ex]
\hline
$\pattern{scale = 0.6}{2}{1/1,2/2}{0/0,0/2,1/0,1/1,1/2}$ & $a_n= n\cdot a_{n-1}-a_{n-2}$, $a_{-1}=0, a_{0}=1$ & $32$ & A058797 \\
\end{tabular}}
\caption{Wilf-classes not containing bivincular patterns, with a known counting
sequence}
\label{tab:notbivincwithseq2}
\end{table}

\subsection{Unenumerated Wilf-classes, not containing bivincular patterns}
Tables~\ref{tab:notbivincularwithoutseq:partI}
and~\ref{tab:notbivincularwithoutseq:partII} show the Wilf-subclasses that we
have not mentioned before and have a unique enumeration sequence up to and
including $S_9$. This shows that these subclasses cannot be merged with any
other Wilf-subclasses and are therefore Wilf-classes.
\begin{table}[ht]
{
\renewcommand{\arraystretch}{1.6}
\begin{tabular}{c|c|l|c}
\multirow{2}{*}{Nr.\ } & \multirow{2}{*}{Repr.\ $p$}  & \multirow{2}{*}{$|S_n(p)|$ for $n = 1, \dotsc, 9$}  & size of \\[-1.75ex]
&&&subclass\\
\hline
\hline
25 & $\pattern{scale=0.6}{2}{1/1,2/2}{2/2,0/0,1/1,0/2}$             & $1, 1, 3, 11, 55, 337, 2437, 20211, 188537$ & $4$\\[0.5ex] %/236
\hline
26 & $\pattern{scale=0.6}{2}{1/1,2/2}{0/0,1/1,2/2}$                 & $1, 1, 2, 7, 35, 218, 1598, 13398, 126157$ & $2$\\[0.5ex] %/109
\hline
27 & $\pattern{scale=0.6}{2}{1/1,2/2}{0/1,2/0,2/2,1/0,1/1,0/2}$     & $1, 1, 4, 18, 99, 631, 4592, 37675, 344809$ & $4$\\[0.5ex] %/438
\hline
28 & $\pattern{scale=0.6}{2}{1/1,2/2}{0/1,1/2,0/0,2/2,1/0,2/1}$     & $1, 1, 3, 16, 94, 613, 4507, 37203, 341817$ & $2$\\[0.5ex] %/390
\hline
29  & $\pattern{scale=0.6}{2}{1/1,2/2}{0/1,1/2,2/0}$                 & $1, 1, 2, 8, 41, 251, 1809, 14986, 139963$ & $4$\\[0.5ex] %/50
\hline
30 & $\pattern{scale=0.6}{2}{1/1,2/2}{0/1,1/2,2/0,1/0,1/1,2/1,0/2}$ & $1, 1, 5, 21, 109, 673, 4797, 38845, 352541$ & $2$\\[0.5ex] %/484
\hline
31 & $\pattern{scale=0.6}{2}{1/1,2/2}{2/0,0/0,0/2,1/1}$             & $1, 1, 3, 11, 56, 349, 2560, 21453, 201545$ & $4$\\[0.5ex] %/234
\hline
32 & $\pattern{scale=0.6}{2}{1/1,2/2}{0/1,1/2,1/0,0/0,2/1}$         & $1, 1, 3, 13, 70, 448, 3307, 27618, 257363$ & $4$\\[0.5ex] %/260
\hline
33 & $\pattern{scale=0.6}{2}{1/1,2/2}{0/1,1/2,2/0,1/0,0/2,2/1}$     & $1, 1, 5, 20, 106, 657, 4707, 38267, 348341$ & $2$\\[0.5ex] %/404
\hline
34 & $\pattern{scale=0.6}{2}{1/1,2/2}{0/1,1/2,0/0,2/2,1/0,1/1,2/1}$ & $1, 1, 4, 20, 107, 664, 4755, 38621, 351151$ & $2$\\[0.5ex] %/476
\hline
35 & $\pattern{scale=0.6}{2}{1/1,2/2}{2/0,0/2,1/1}$                 & $1, 1, 3, 11, 53, 315, 2217, 17990, 165057$ & $2$\\[0.5ex] %/123
\hline
36 & $\pattern{scale=0.6}{2}{1/1,2/2}{0/1,1/2,0/0,1/0,1/1,2/1}$     & $1, 1, 3, 14, 76, 480, 3491, 28792, 265708$ & $4$\\ %/392
\end{tabular}
}
\caption{Wilf-classes not containing bivincular patterns, without a known
counting sequence -- Part I.}
\label{tab:notbivincularwithoutseq:partI}
\end{table}

\begin{table}[ht]
{
\renewcommand{\arraystretch}{1.6}
\begin{tabular}{c|c|l|c}
\multirow{2}{*}{Nr.\ } & \multirow{2}{*}{Repr.\ $p$}  & \multirow{2}{*}{$|S_n(p)|$ for $n = 1, \dotsc, 9$}  & size of \\[-1.75ex]
&&&subclass\\
\hline
\hline
38 & $\pattern{scale=0.6}{2}{1/1,2/2}{0/1,1/2,1/0,0/2,2/1}$         & $1, 1, 4, 15, 80, 493, 3565, 29279, 269621$ & $4$\\[0.5ex] %/274
\hline
39  & $\pattern{scale=0.6}{2}{1/1,2/2}{0/1,1/0}$                     & $1, 1, 2, 5, 17, 71, 357, 2101, 14203$ & $8$\\[0.5ex] %/17
\hline
40 & $\pattern{scale=0.6}{2}{1/1,2/2}{0/1,1/2,1/0,2/1}$             & $1, 1, 3, 10, 50, 290, 2018, 16023, 143601$ & $2$\\[0.5ex] %/140
\hline
41 & $\pattern{scale=0.6}{2}{1/1,2/2}{0/1,2/0,0/2,1/1}$             & $1, 1, 3, 12, 61, 376, 2715, 22416, 207950$ & $12$\\[0.5ex] %/179
\hline
42 & $\pattern{scale=0.6}{2}{1/1,2/2}{2/2,2/0,0/0,0/2,1/1}$         & $1, 1, 4, 15, 80, 501, 3666, 30467, 283196$ & $2$\\[0.5ex] %/373
\hline
43  & $\pattern{scale=0.6}{2}{1/1,2/2}{0/1,1/0,2/2}$                 & $1, 1, 2, 7, 35, 217, 1586, 13287, 125237$ & $4$\\[0.5ex] %/72
\hline
44 & $\pattern{scale=0.6}{2}{1/1,2/2}{0/1,1/0,2/2,1/1,0/2}$         & $1, 1, 3, 13, 69, 437, 3209, 26751, 249329$ & $8$\\[0.5ex] %/325
\hline
45 & $\pattern{scale=0.6}{2}{1/1,2/2}{0/1,1/2,1/0,1/1,2/1,0/2}$     & $1, 1, 4, 16, 83, 512, 3671, 29983, 274757$ & $4$\\[0.5ex] %/408
\hline
46 & $\pattern{scale=0.6}{2}{1/1,2/2}{0/1,1/0,0/2,2/2}$             & $1, 1, 3, 12, 62, 387, 2819, 23409, 217949$ & $24$\\[0.5ex] %/180
\hline
47 & $\pattern{scale=0.6}{2}{1/1,2/2}{0/1,1/2,2/1,2/0}$             & $1, 1, 3, 12, 62, 385, 2789, 23040, 213566$ & $16$\\ %/138
\end{tabular}
}
\caption{Wilf-classes not containing bivincular patterns, without a known
counting sequence -- Part II. Only subclass $39$ is recognized by the OEIS, as
A101900.}
\label{tab:notbivincularwithoutseq:partII}
\end{table}

Although we are unable to provide a formula for the enumeration of these
classes, we do have a conjecture for one of them.
\begin{conjecture}
The following pattern is a representative pattern for Wilf-subclass $39$ in
Table~\ref{tab:notbivincularwithoutseq:partII}.
\[
p = \pattern{scale = 0.7}{2}{1/1,2/2}{0/1,1/0}
\]
The number of permutations of length $n$ avoiding the pattern is the same as the
absolute value of the $n$-th line in column $0$ of a triangular matrix given by
the formula
\[
T(n, k) = \sum_{j=0}^{n-k} T(n-k, j)\cdot T(j+k-1, k-1)\footnote{http://oeis.org/A101900}
\]
for $n\geq k >0$ with $T(0, 0)=1$ and $T(n, 0) = -\sum_{j=1}^{n} T(n, j)$ for
$n > 0$.
\end{conjecture}

\subsection{The remaining subclasses}
In this section we will list all the remaining subclasses. As can be seen in
Table~\ref{tab:remainig} some of these subclasses have the same enumeration
sequence for $n\leq 9$. We believe those subclasses can be merged so that the
final number of Wilf-classes will be $46$.
\begin{table}[ht]
{
\renewcommand{\arraystretch}{1.6}
\begin{tabular}{c|c|l|c}
\multirow{2}{*}{Nr.\ } & \multirow{2}{*}{Repr.\ $p$}  & \multirow{2}{*}{$|S_n(p)|$ for $n = 1, \dotsc, 9$}  & size of \\[-1.75ex]
&&&subclass\\
\hline
\hline
48 & $\pattern{scale = 0.6}{2}{1/1,2/2}{0/1,1/2,0/0,2/1,2/2}$ & $1, 1, 2, 9, 54, 370, 2849, 24483, 232913$ & $8$\\ %/259
49 & $\pattern{scale = 0.6}{2}{1/1,2/2}{0/1,1/2,0/0,1/1,2/0}$ & $1, 1, 2, 9, 54, 370, 2849, 24483, 232913$ & $8$\\ %/268
50 & $\pattern{scale = 0.6}{2}{1/1,2/2}{0/1,1/2,0/0,1/1,2/2}$ & $1, 1, 2, 9, 54, 370, 2849, 24483, 232913$ & $4$\\[0.5ex] %/270
\hline
51 & $\pattern{scale = 0.6}{2}{1/1,2/2}{0/1,2/0,0/0,1/1,2/2}$ & $1, 1, 3, 12, 64, 412, 3074, 25946, 243996$ & $8$\\ %/309
52 & $\pattern{scale = 0.6}{2}{1/1,2/2}{0/1,2/0,2/2,0/2,1/1}$ & $1, 1, 3, 12, 64, 412, 3074, 25946, 243996$ & $8$\\[0.5ex] %/323
\hline
53 & $\pattern{scale = 0.6}{2}{1/1,2/2}{0/1,1/2,0/0,2/1}$ & $1, 1, 2, 8, 43, 277, 2070, 17567, 166648$ & $8$\\ %/131
54 & $\pattern{scale = 0.6}{2}{1/1,2/2}{0/1,0/0,1/1,2/2}$ & $1, 1, 2, 8, 43, 277, 2070, 17567, 166648$ & $8$\\[0.5ex] %/165
\hline
55 & $\pattern{scale = 0.6}{2}{1/1,2/2}{0/1,1/2,0/0,2/0,1/1,2/1}$ & $1, 1, 3, 15, 85, 549, 4043, 33559, 310429$ & $8$\\ %/389
56 & $\pattern{scale = 0.6}{2}{1/1,2/2}{0/1,1/2,0/0,2/2,1/1,2/1}$ & $1, 1, 3, 15, 85, 549, 4043, 33559, 310429$ & $8$\\[0.5ex] %/391
\hline
57 & $\pattern{scale = 0.6}{2}{1/1,2/2}{0/1,1/2,1/1,2/0}$ & $1, 1, 2, 8, 42, 265, 1956, 16482, 155739$ & $4$\\ %/148
58 & $\pattern{scale = 0.6}{2}{1/1,2/2}{0/1,1/0,1/1,2/2}$ & $1, 1, 2, 8, 42, 265, 1956, 16482, 155739$ & $4$\\[0.5ex] %/186
\hline
59 & $\pattern{scale = 0.6}{2}{1/1,2/2}{0/1,1/2,0/0,1/1,2/1}$ & $1, 1, 2, 9, 52, 341, 2540, 21360, 200536$ & $8$\\ %/261
60 & $\pattern{scale = 0.6}{2}{1/1,2/2}{0/1,0/0,1/1,2/1,2/2}$ & $1, 1, 2, 9, 52, 341, 2540, 21360, 200536$ & $4$\\[0.5ex] %/300
\hline
61 & $\pattern{scale = 0.6}{2}{1/1,2/2}{0/1,1/2,0/0,2/0}$ & $1, 1, 2, 8, 44, 290, 2204, 18930, 181120$ & $8$\\ %/133
62 & $\pattern{scale = 0.6}{2}{1/1,2/2}{0/1,1/0,0/0,2/2}$ & $1, 1, 2, 8, 44, 290, 2204, 18930, 181120$ & $4$\\[0.5ex] %/164
\hline
63 & $\pattern{scale = 0.6}{2}{1/1,2/2}{0/1,1/2,0/0,2/1,2/0}$ & $1, 1, 3, 13, 71, 461, 3447, 29093, 273343$ & $8$\\ %/258
64 & $\pattern{scale = 0.6}{2}{1/1,2/2}{0/1,1/2,2/0,0/2,1/1}$ & $1, 1, 3, 13, 71, 461, 3447, 29093, 273343$ & $4$\\ %/284
65 & $\pattern{scale = 0.6}{2}{1/1,2/2}{0/1,1/0,0/0,1/1,2/2}$ & $1, 1, 3, 13, 71, 461, 3447, 29093, 273343$ & $4$\\ %/311
\end{tabular}
}
\caption{Remaining subclasses}
\label{tab:remainig}
\end{table}
We have only been able to find formula for the enumeration of one subclass.

\begin{proposition}\label{prop:11}
The permutations avoiding the pattern
$p = \pattern{scale = 0.5}{2}{1/1,2/2}{0/1,1/2,2/0,0/2,1/1}$ are the connected
permutations, thus
\[
\left| S_n\left(\pattern{scale = 0.6}{2}{1/1,2/2}{0/1,1/2,2/0,0/2,1/1}\right) \right| = [x^n]\left(1-\frac{1}{ \sum_{n} n!x^n }\right).
\]
\end{proposition}

To be able to provide that proof we have to consider invariant sets and
connected permutations.
\begin{definition}
If a permutation sends letters $1,2,\ldots,j$ to $1,2,\ldots,j$, where $0<j<n$,
we say that $\{1,2,\ldots,j\}$ is an \emph{invariant set}. Permutations that do
not contain an invariant set are called \emph{connected} permutations.
\end{definition}
According to Aguiar and Sottile in~\cite{MR2103213} the number of permutations
of length $n$ with no global descents\footnote{A global descent in a permutation
$\pi$ of length $n$ is an index $k < n$ such that $\pi_i > \pi_j$ for all
$i \leq k$ and $j \geq i+1$.} is given by
\[
1-\frac{1}{ \sum_{n} n!x^n },
\]
which is the same as the number of connected permutations of length $n$.

\begin{proof}[Proof of Proposition~\ref{prop:11}]
We will show the contrapositive, i.e., that the pattern $p$ occurs in a
permutation $\pi$ if and only if $\pi$ has an invariant set.

In the pattern $p=\pattern{scale = 0.5}{2}{1/1,2/2}{0/1,1/2,2/0,0/2,1/1}$, let
us call the points in the pattern $v$ and $w$, respectively. If $p$ occurs in a
permutation $\pi$, then there is no letter to the left of $w$ that is greater
than $v$. Also, there can be no smaller letter than $v$ to the right of $w$.
Hence, if there exist letters $a_1,a_2,\ldots,a_k$ to the left of $w$, then
$a_1,a_2,\ldots,a_k$ are all smaller than $v$.

On one hand, if $\{a_1,a_2,\ldots,a_k\}=\varnothing$, then $v$ is both the
leftmost and the smallest letter in $\pi$. Then $\{v\}$ is an invariant set. On
the other hand, if $\{a_1,a_2,\ldots,a_k\}\not=\varnothing$ then the points
$a_1,a_2,\ldots,a_k$ are in the first $k$ positions in $\pi$ and hence
$\{a_1,a_2,\ldots,a_k,v\}$ is an invariant set.

If $\pi$ has an invariant set, we know that the lowest $k$ letters in the
permutation are in the first $k$ positions. Let us choose the highest letter in
the invariant set, and call it $v$. Then we choose $w$ in position $k+1$. Now,
$v$ and $w$ form the pattern $p$.
\end{proof}

\section{Open questions}
We end with three open questions:

\begin{enumerate}
  \item What is the actual number of Wilf-classes for mesh patterns of length $2$? As mentioned above
  it is at most $56$ but based on computer experiments we conjecture it to be $46$.
  \item Suppose $p$ and $q$ are mesh patterns such that $|S_n(p)| = |S_n(q)|$ for all $1 \leq n \leq N$;
  how large does $N$ have to be (as a function of the length of the patterns $p$ and $q$) to guarantee that
  $p$ and $q$ are Wilf-equivalent?
  \item Is there a stronger version of the Shading Lemma (Lemma \ref{lem:shading}) that explains more pattern
  coincidences, perhaps strong enough to give all the coincidences between the patterns in Table~\ref{tab:classical1}.
\end{enumerate}

\section*{Acknowledgments}
We thank Einar Steingrimsson for helpful criticism and comments. Furthermore, we
are grateful to Sergey Kitaev, Anders Claesson and Anna Ingolfsdottir for the
useful remarks, and David Callan for a helpful explanation on a particular
pattern. Finally Hjalti Magnusson for all the technical support with
typesetting.

\appendix

\section{Switch operations} \label{app:switch}
\begin{definition}
We define a set of operations called \emph{switch operations}. Each operation
takes a pattern $p$ and breaks it down into two parts by the largest letter $n$.
The first part contains the part of the pattern $p$ that appears before $n$. The
second part contains the part of $p$ that appears after $n$. Then one of the two
operations, $\id,r$, is used on each part, where $(\tau, R)^{\id} = (\tau, R)$.
After that theses parts can be switched, i.e., part two will then appear before
$n$ and part one after $n$. An operation will be denoted as $S_{a,b,d}$, Where
$a$ is the operation used on part one and $b$ the operation used on part two.
The letter $d$ is $1$ if part one and two are switched and $0$ otherwise. These
switch operations can also be used on permutations.
\end{definition}

\begin{example}
This example shows the effect of the switch operation $S_{r,r,1}$ on a mesh
pattern of length $6$.
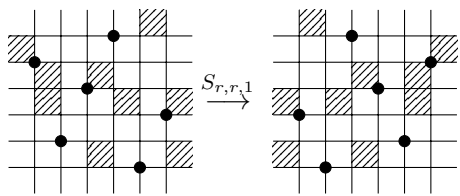
\[
\pattern{scale = 1}{6}{1/5,2/2,3/4,4/6,5/1,6/3}{0/5,1/3,1/4,3/1,3/4,4/3,5/6,6/1,6/3}
\overset{S_{r,r,1}}{\longrightarrow}
\pattern{scale = 1}{6}{1/3,2/1,3/6,4/4,5/2,6/5}{0/1,0/3,1/6,2/3,3/1,3/4,5/3,5/4,6/5}
\]
\end{example}

\begin{observation}\label{obs:switch}
Let $p=(\tau, R)$ be a pattern of length $n$ where the top line is shaded, that
is $\{(0,n), (1,n), \ldots, (n,n)\}\subseteq R$. Then a permutation $\pi$ avoids
$p$ if and only if $\pi^{S_{a,b,d}}$ avoids $p^{S_{a,b,d}}$. That is, the switch
operations preserve Wilf-equivalence for this kind of pattern.
\end{observation}

\begin{example}
Here is another example showing the effect of the switch operation $S_{r,\id,1}$
on a mesh pattern of length $6$.
\[
\pattern{scale = 1}{6}{1/3,2/1,3/6,4/2,5/4,6/5}{0/0,1/4,2/3,3/4,5/2,5/3,6/5,1/6,2/6,3/6,5/6,6/6,0/6,4/6}
\overset{S_{r,\id,1}}{\longrightarrow}
\pattern{scale = 1}{6}{1/5,2/4,3/2,4/6,5/3,6/1}{0/5,0/6,1/2,1/3,1/6,3/6,4/0,5/4,5/6,6/3,6/6,2/6,4/6}
\]
According to Observation~\ref{obs:switch} theses two patterns are Wilf-equivalent.
\end{example}

\bibliographystyle{amsplain}
\bibliography{biblio}

\end{document}